\documentclass{amsart}

\usepackage{amsfonts}
\usepackage{enumerate}
\usepackage[usenames,dvipsnames]{color}
\usepackage{amscd}
\usepackage{nextpage}
\usepackage{stmaryrd}
\usepackage{url}

\newtheoremstyle{theorem}{}{20pt}{\normalfont}{0pt}{\bfseries}{.}
            {0.8pc}{\thmname{#1}\thmnumber{ #2}\thmnote{ \textup{(#3)}}}
            
\newtheorem{theorem}{Theorem}[section]

\newtheorem{thm}[theorem]{Theorem}
\newtheorem{prop}[theorem]{Proposition}
\newtheorem{cor}[theorem]{Corollary}
\newtheorem{assume}[theorem]{Assumption}

\theoremstyle{definition}

\newtheorem{example}[theorem]{Example}

\theoremstyle{remark}
\newtheorem{remark}[theorem]{Remark}

\numberwithin{equation}{section}
            
\newcommand{\bcX}{\mathcal{X}}
\newcommand{\bcY}{{\mathcal{Y}}}
\newcommand{\Tsch}{{Chebyshev }}

\newcommand{\Ylangle}{{_{\cY'}}\langle}

\newcommand{\Yrangle}{\rangle_{\cY}}
\newcommand{\Indx}{\mathcal J}

\DeclareMathOperator{\supp}{supp}

\newcommand{\R}{\mathbb{R}}

\newcommand{\N}{\mathbb{N}}

\newcommand{\C}{\mathbb{C}}

\newcommand{\E}{\mathbb{E}}

\newcommand{\bsb}{{\boldsymbol b}}
\newcommand{\bsi}{{\boldsymbol i}}
\newcommand{\bsy}{{\boldsymbol y}}
\newcommand{\bsz}{{\boldsymbol z}}
\newcommand{\bsx}{{\boldsymbol x}}
\newcommand{\bsv}{{\boldsymbol v}}
\newcommand{\bsA}{{\boldsymbol \Phi}}
\newcommand{\bnu}{{\boldsymbol \nu}}
\newcommand{\bmu}{{\boldsymbol \mu}}

\newcommand{\bbeta}{{\boldsymbol \beta}}
\newcommand{\bsrho}{{\boldsymbol \rho}}

\newcommand{\eps}{\varepsilon}
\newcommand{\Tscheb}{Chebyshev }

\newcommand{\cD}{{\mathcal D}}

\newcommand{\cF}{{\mathcal F}}

\newcommand{\cL}{{\mathcal L}}

\newcommand{\cX}{{\mathcal X}}
\newcommand{\cY}{{\mathcal Y}}

\newcommand{\fa}{{\mathfrak A}}

\newcommand{\bPhi}{{\boldsymbol \Phi}}

\newcommand{\A}{\Phi}

\newcommand{\triple}{\hspace{.5mm} | \hspace{-0.5mm} | \hspace{-0.5mm}  | \hspace{.5mm}}

\newcommand{\bas}{\begin{align*}}
\newcommand{\eas}{\end{align*}}
\newcommand{\ba}{\begin{align}}
\newcommand{\ea}{\end{align}}
\newcommand{\bes}{\begin{equation*}}
\newcommand{\ees}{\end{equation*}}
\newcommand{\be}{\begin{equation}}
\newcommand{\ee}{\end{equation}}

\newcommand{\KL}{{\rm Karh\'{u}nen-Loeve }}

\begin{document}

\title{
Compressive sensing Petrov-Galerkin approximation
\\ ~of high-dimensional parametric operator equations
}
\date{\today}

\author{Holger Rauhut}
\address{Lehrstuhl C f{\"u}r Mathematik (Analysis), RWTH Aachen University, Pontdriesch 10, 52062 Aachen, Germany}
\curraddr{}
\email{rauhut@mathc.rwth-aachen.de}
\thanks{}

\author{Christoph Schwab}
\address{Seminar for Applied Mathematics, ETH Zurich, R{\"a}mistrasse 101, 8092 Z{\"u}rich, Switzerland}
\curraddr{}
\email{christoph.schwab@sam.math.ethz.ch}
\thanks{}

\subjclass[2010]{Primary 35B30 Secondary: 65N30, 41A58, 94A20}

\begin{abstract}
We analyze the convergence of compressive sensing based
sampling techniques for the efficient evaluation of 
functionals of solutions for a class of 
high-dimensional, affine-parametric, linear operator equations
which depend on possibly infinitely many parameters.
The proposed algorithms are based on
so-called ``non-intrusive''
sampling of the high-dimensional parameter
space, reminiscent of Monte-Carlo sampling.
In contrast to Monte-Carlo, however, a functional of the parametric solution
is then computed via compressive sensing methods from samples of functionals of
the solution. 
A key ingredient in our analysis of independent interest
consists in a generalization of recent results on the approximate 
sparsity of generalized polynomial chaos representations (gpc)
of the parametric solution families,
in terms of the gpc series with respect to tensorized \Tscheb polynomials.
In particular, we establish sufficient conditions on the 
parametric inputs to the parametric operator equation
such that the \Tscheb coefficients of gpc expansion are contained 
in certain weighted $\ell_p$-spaces for $0<p\leq 1$.
Based on this we show that
reconstructions of the parametric solutions computed from
the sampled problems converge, with high probability,  
at the $L_2$, resp.\ $L_\infty$ convergence rates afforded 
by best $s$-term approximations of the parametric solution up to logarithmic factors.
\end{abstract}
\maketitle

Key Words: 
Compressive sensing,
affine-parametric operator equations, 
parametric diffusion equation,
$s$-term approximation, high-dimensional approximation,
Tensorized Chebyshev polynomial chaos approximation.

\allowdisplaybreaks

\section{Introduction}
\label{sec:Intro}

The numerical solution of parametric operator equations 
on high-dimensional parameter spaces
has recently attracted substantial interest,
in particular in uncertainty quantification. 
There, one often 
models uncertainty in input parameters probabilistically:
randomness in the coefficients of 
a partial differential equation
may account for the fact that the true model, i.e., the true coefficient,
is in practice not known exactly. 
Having a full parametric solution of the equation at hand, 
it is straightforward to (re)insert the probabilistic model
and to compute quantities such as the expected solution, covariances and higher moments.
Unfortunately, the parameter domain is often high- or even infinite-dimensional, 
making it computationally hard to approximate
the solution well with classical approaches due to the curse of dimension. 
On the other hand, the classical Monte Carlo method for estimating
the expected solution, say, of a PDE with random input converges at best 
at a rate of $m^{-1/2}$ for $m$ sample evaluations 
(with constants that are independent of the dimension).

Due to the importance of this problem class,
a number of constructive computational approaches have emerged 
in recent years which alleviate or even overcome 
the curse of dimensionality.
A first class of methods can be described as ``greedy'', adaptive
deterministic approximations. By this, we mean that a sequence of 
parametric approximations of the parametric operator equation is computed
\emph{sequentially} by successive numerical solution of instances 
of the parametric operator equation. 
We mention here adaptive stochastic Galerkin methods 
\cite{EGSZ14_1230,EGSZ14_1045,G13_531},
reduced basis approaches (see, eg., \cite{BCDDPW11,BufMadPat12}),
abstract greedy approximation in Banach spaces
and adaptive Smolyak discretizations \cite{SS13_813,SS14_1117}
as well as adaptive interpolation methods (see \cite{CCS2014} and the references there)
and sampling methods \cite{TangIacc2014}.
While adaptive Galerkin discretizations as in \cite{EGSZ14_1230,EGSZ14_1045,G13_531}
are intrusive, the other mentioned approaches are not; 
they are, however, \emph{sequential}
in the sense that they rely on \emph{successive} 
numerical solution of the operator equation on parameter instances.
This is in contrast to, say, (multilevel) Monte-Carlo \cite{MSS}, 
or Quasi-Monte Carlo approaches \cite{DKLS14_1319}
which likewise offer dimension-independent convergence rates for statistical
moments of the solution, and which allow to access the parametric solution
\emph{simultaneously} at a set of samples. These methods do not, however,
allow \emph{recovery} of the parametric solution, 
but only compute statistical quantities
such as the mean solution. Further methods include
sparse, anisotropic collocation methods \cite{NTW}, 
sparse grids \cite{BGActa,GerstGrieb03,notewe08}, 
low rank tensor methods \cite{KhSc} and least squares approximation 
methods based on sample evaluation at randomly chosen parameters \cite{CDL2013,MiNoSchT2014}.

In this article, we introduce and analyze a method for approximating 
functionals of the solution of 
affine parametric elliptic operator equations 
based on ideas from compressive sensing \cite{fora13} by exploiting that the solution 
can be shown to be well-approximated by a sparse expansion
in terms of tensorized \Tscheb polynomials under certain natural 
assumptions \cite{chalch14,codesc10-1,codesc11,Hanse486,Kuo475,ScMCQMC12}. 
Samples of the solutions for specific parameter choices
are computed via Petrov-Galerkin methods.
Recent work related to compressive sensing approaches for 
approximating high-dimensional,
parametric PDEs includes \cite{doow11,pehado14,sasanaderith14,yaka13} 
where also rather detailed numerical experiments were performed.
They indicate that CS-based sampling of the parameter space
(such as developed here) does, indeed, capture rather closely
near optimal $s$-term polynomial chaos approximations 
of the parametric solution. 
However, in contrast to these earlier contributions 
where only a partial analysis was performed, 
we provide rigorous convergence rate bounds which equal, 
up to $\log$-terms, the best $s$-term bounds
for the full discretization, i.e., 
with compressed sensing in parameter space and with a
stable Petrov-Galerkin discretization of the operator.
We use the acronym CSPG for this compressed-sensing Petrov-Galerkin
discretization of parametric operator equations.

Specifically, the idea of the method is to exploit 
(approximate) sparsity of the coefficient sequence in the polynomial chaos expansion
of the parametric solution in terms of \Tscheb polynomials. 
Suppose we have a numerical evaluation of (a functional of) 
the solution at $m$ parameter points.
Reducing to a finite
polynomial chaos expansion (on some possibly large index set), 
the (unknown) \Tscheb coefficient sequence satisfies an underdetermined 
linear system of equations. 
Compressive sensing methods are able to solve such 
underdetermined systems under certain conditions knowing 
that the solution is (approximately) sparse.
This paper aims at providing a detailed analysis of this idea.

The present exposition's focus is on analytical foundations
of the CSPG discretization.
First numerical experiments, confirming and extending the present 
theory, are available in \cite{BBRS15_2095}.
Compressive sensing approaches feature the 
following advantages in comparison to 
deterministic collocation or sampling strategies,
as described in \cite{NTW,CCS2014}.
\begin{itemize}
\item {\bf Parallelizable.} 
The samples of the parameter are taken at random and in advance. 
For this reason, solutions corresponding to these parameter 
samples can easily be computed in parallel,
unlike, for example, strategies which adaptively (``greedily'') 
determine subsequent sampling
points based on previous computations, see e.g.\ \cite{BCDDPW11}.
\item {\bf Nonintrusive.} 
As just said, the method works with solution samples, 
and for computing these any standard method can be applied 
(that guarantees a certain error bound).
No additional implementation is required on the level of the operator equation. 
One only has to add the compressive sensing part which operates independently
of the PDE solver. 
In particular, the CSPG approach is potentially 
very efficient if the PDE solves are significantly more expensive
in comparison to a compressive sensing reconstruction.
\item {\bf Rigorous error bounds.} 
We derive  error bounds for the CSPG discretization
which provide near-optimal convergence rates in terms of the number of computed samples.
\item {\bf Mild structural requirements on index sets.}
Sparse adaptive collocation methods and sparse-grid approximations 
as e.g. in \cite{CCS2014,SS14_1117,chalch14} require structural conditions
on the sets of best $N$-term polynomial chaos approximations, 
such as downward closedness. 
It is known that optimal index sets may not always satisfy
these conditions (see, e.g. \cite{TWZ14}).
The present CSPG approach does not require such conditions
and can capture rather general index sets arising in
$N$-term approximations.
\end{itemize}
Working with random parameter choices, 
the method analyzed here is reminiscent of Monte Carlo methods at first sight. 
However, once the samples are computed, they are combined
very differently as in Monte Carlo methods in order to compute (functionals of) the full parametric solution. 
Moreover, Monte Carlo methods usually only compute expectations or moments rather than
(functionals of) full parametric solutions. 
The convergence rate of MC methods is, intrinsically, limited to $1/2$ by the central limit theorem,
however with constants that are independent of the dimension of the parameter
 domains where the MC sampling takes place. Higher order convergence rates
have been shown recently for certain types of 
quasi Monte-Carlo Galerkin methods in \cite{DKLN14_1134}.
Such methods only aim at computing the integrals of (functionals of) the
solution rather than (functionals of) the full parametric solution.
As we show in the present work, 
convergence rates which are limited only by 
best $s$-term approximation rates of polynomial chaos expansions 
of solutions to countably-parametric operator equations
are achieveable by compressed sensing with 
random sampling from
the probability measure induced by the weight function
of the polynomial chaos -- in our case, the product \Tscheb measure.

\medskip

Ahead, we consider linear,
affine parametric operator equations of the form 
$A(\bsy) u(\bsy) = f$ with 
$$
A(\bsy) = A_0 + \sum_{j \geq 1} y_j A_j,
$$
where the $A_j$ are bounded, linear 
operators and $\bsy = (y_j)_{j\geq 1}$ with
$y_j \in [-1,1]$, $j=1,2,\hdots,$ is a countable sequence of parameters.
For the sake of concreteness, we will now illustrate our ideas and findings for the example of a 
parametric diffusion equation on a bounded Lipschitz domain $D \subset \R^n$ 
(where one should think of $n=1,2,3$). 
For a diffusion coefficient that depends affinely on a parameter 
sequence $\bsy$, such as
\be
\label{eq:aKL}
a(x,\bsy) = \bar{a}(x) + \sum_{j\geq 1} y_j \psi_j(x), \quad x\in D,
\ee
we consider the parametric, elliptic problem
\be \label{eq:Diffusion}
A(\bsy) u := - \nabla \cdot(a(\cdot,\bsy) \nabla u) = f 
\quad {\mbox{ in }} D, 
\quad u|_{\partial D} = 0.
\ee
The expansion \eqref{eq:aKL} may, for instance, 
arise from a \KL 
decomposition of a random field.
The weak formulation of \eqref{eq:Diffusion} in the Sobolev space $V:=H_0^1(D)$ reads:
Given $f\in V^*$, for every $\bsy\in U:= [-1,1]^\N$, find $u(\bsy)\in V$ such that
\be\label{diff:eq:weak}
\int_D a(x,\bsy) \nabla u(x) \cdot \nabla v(x) dx 
= 
\int_D f(x) v(x) dx \quad \mbox{ for all } v \in V.
\ee
We require the {\em uniform ellipticity assumption}:
there exist constants $0<r \leq R <\infty$ such that
\be\label{UEA}
r \leq a(x,\bsy) \leq R \qquad \mbox{ for almost all } x \in D, \mbox{ for all } \bsy \in U .
\ee
The Lax-Milgram Lemma then ensures that for every $\bsy \in U$,
\eqref{diff:eq:weak} has a unique solution $u(\cdot, \bsy) \in V$, which satisfies the a priori estimate
\[
\sup_{\bsy\in U} 
\| u(\bsy) \|_V \leq r^{-1} \|f\|_{V^*}.
\]
Given a bounded linear functional $G : V \to \R$, 
we are interested in a numerical approximation of the function
\[
F(\bsy) := G(u(\bsy)), \qquad \bsy \in U = [-1,1]^\N.
\]
To this end, we consider a high-dimensional \Tscheb expansion of $u(\bsy)$. We introduce the $L^2$-normalized univariate \Tscheb polynomials
\be\label{eq:DefTj}
T_j(t) = \sqrt{2} \cos(j \arccos t), \quad j \in \N, \quad \mbox{ and } \quad T_0(t) \equiv 1. 
\ee
These functions are orthonormal with respect to the probability measure 
$\sigma$ on $[-1,1]$, defined as
\[
d\sigma(t) = \frac{1}{\pi\sqrt{1-t^2}} dt,
\]
that is, $\int_{-1}^1 T_j(t) T_k(t) d \sigma(t) = \delta_{jk}$, $j,k \in \N_0$.
Let $\cF$ be the infinite-dimensional set of multi-indices with finite support
\begin{equation}\label{def:cF}
\cF 
:= 
\{ \bnu = (\nu_1,\nu_2,\hdots), \nu_j \in \N_0 \mbox{ and } \nu_j \neq 0 
\mbox{ for at most finitely many } j 
\}.
\end{equation}
For $\bnu \in \cF$, the tensorized \Tscheb polynomial $T_{\bnu}$ 
is defined as
\begin{equation}\label{tensCheb}
T_{\bnu}(\bsy) = \prod_{j = 1}^\infty T_{\nu_j}(y_j) = \prod_{j \in \supp \bnu} T_{\nu_j}(y_j), 
\quad \bsy = (y_j)_{j\geq 1} \in U = [-1,1]^\N,
\end{equation}
where the product has only finitely many nontrivial factors 
due to \eqref{eq:DefTj} and the definition of $\cF$. 
With the product probability measure
\begin{equation}\label{prod:Cheb:meas}
\eta  
=
\bigotimes_{j \geq 1} \frac{d y_j}{\pi \sqrt{1-y_j^2}}, \quad \bsy \in U,
\end{equation}
the functions $T_{\bnu}$ are orthonormal, i.e., 
for every $\bnu,\bmu \in \cF$,
\begin{equation}\label{Tnu:orth}
\int_U T_{\bnu}(\bsy) T_{\bmu}(\bsy) d\eta(\bsy) 
= 
\delta_{\bnu,\bmu} := \left\{ \begin{array}{l} 1 \quad \mbox{if} \;\;\bnu=\bmu,
\\
0 \quad \mbox{else}.
\end{array}
\right.
\end{equation}
In fact, they form an orthonormal basis for $L_2(U,\eta)$. Here, as usual, $L_p(U,\eta)$, denotes the Lebesgue space of all $p$-integrable 
functions on $U$ endowed with the norm 
$\|F\|_{L_p(U,\eta)} = \left( \int_U |F(\bsy)|^p d\eta(\bsy)\right)^{1/p}$ 
for $0<p<\infty$, with the usual modification for $p=\infty$.
We write the solution $u = u(\bsy)$ of \eqref{diff:eq:weak} as a so-called polynomial chaos expansion 
\begin{equation}\label{Cheb:expand:sol}
u(\bsy) = \sum_{\bnu \in \cF} d_{\nu} T_{\nu}(\bsy)
\end{equation}
with $d_{\nu} \in V = H_0^1(D)$. Generalizing initial contributions in \cite{codesc10-1,codesc11,Kuo475,ScMCQMC12}, where Taylor and Legendre expansions were considered, it was shown 
in \cite{Hanse486} that 
\begin{equation}
\| (\|d_{\bnu} \|)_{\bnu \in \cF} \|_p := \left( \sum_{\bnu \in \cF} \|d_{\bnu}\|_V^p\right)^{1/p} < \infty
\end{equation}
for some $0 < p \leq 1$ provided that 
\begin{equation}\label{cond_psi_lp}
\| ( \| \psi_j\|_\infty )_{j \geq 1} \|_p = \left( \sum_{j \geq 1} \|\psi_j\|_\infty^p \right)^{1/p} < \infty
\end{equation}
holds in addition to \eqref{UEA}. 
In order to simplify the discussion, we now consider a bounded linear functional 
$G: V \to \R$ applied to the solution,
\[
F(\bsy) := G(u(\bsy)), \qquad \bsy \in U. 
\]
By linearity and boundedness, we can write
\[
F(\bsy) = \sum_{\bnu \in \cF} g_{\bnu} T_{\bnu}(\bsy)
\]
with $g_{\bnu} = G(d_{\bnu}) \in \R$ and $|g_{\bnu}| \leq \| G \|\, \|d_{\bnu} \|_V$ 
so that $(g_{\bnu})_{\bnu \in \cF} \in \ell_p(\cF)$ under condition \eqref{cond_psi_lp}.
A bound due to Stechkin, see e.g. \cite[Theorems 2.3 \& 2.5]{fora13}, then implies that 
the map $\bsy \mapsto F(\bsy)$
can be well-approximated by a sparse \Tscheb expansion
with $s$ terms, see also \cite{codesc10-1,codesc11}. 
More precisely, 
for $0<p<2$ and for every $s\in \N$
there exists a finite index set
$S \subset \cF$ with $\#S = s$ such that
\begin{equation}\label{best:approx:rate}
\| F - \sum_{\bnu \in S} g_{\bnu} T_{\bnu} \|_{L_2(U,\eta)} \leq s^{1/2-1/p} \|\mathbf{g}\|_p.
\end{equation}
In particular, for $p$ close to $0$, the approximation error 
tends to $0$ with increasing $s$ 
at a rate which is determined only by $p$ 
(the ``compressibility'' of the input data), 
independent of the number of ``active'' variables in the approximation 
and, therefore, free from the curse of dimensionality.

The fact that $F(\bsy)$ is well-approximated by a sparse expansion suggests 
to use compressive sensing for the reconstruction of $F$ from a small
number of samples $F(u(\bsy_1)), \hdots, F(u(\bsy_m))$. 
Assume that we know a priori a finite set $\Indx_0 \subset \cF$ 
for which we are sure that
it contains the support set $S$ corresponding to the best approximation. 
We allow $\Indx_0$ to be significantly larger than the optimal set
$S \subset \Indx_0$. 
Ideally, its size scales polynomially (or at least subexponentially) in $s$. 
Unlike, for example, the
adaptive collocation methods in \cite{CCS2014,chalch14}, the presently proposed 
CSPG approach does not require structural properties (such as ``downward closedness'')
of the set $S$ for its feasibility and optimality (albeit only ``with high probability'').
We will provide explicit descriptions of possible sets $\Indx_0 = \Indx_0^s$ 
later on, see \eqref{def:Indxs},
and estimate their size in Corollary~\ref{cor:size:Indx}.
Then the map $\bsy \mapsto F(\bsy)$
can be well-approximated by a tensorized \Tscheb expansion
with coefficients from $\Indx_0$, i.e.,
\begin{equation}\label{eq:TensTsch}
F(\bsy) \approx \sum_{\bnu \in \Indx_0} g_{\bnu} T_{\bnu}(\bsy).
\end{equation}
Given parameter choices $\bsy_1,\hdots,\bsy_m$, the samples 
$b_\ell := F(\bsy_\ell) = G(u(\bsy_{\ell}))$, $\ell=1,\hdots,m$, 
can be computed (at least approximately) via numerically 
solving the associated diffusion equation and applying the functional $G$. 
With the sampling matrix $\bPhi \in \R^{m \times N}$, $N = \#\Indx_0$,
which is defined component-wise as
\[
\Phi_{\ell,\bnu} = T_{\bnu}(\bsy_\ell), \quad \ell = 1,\hdots,m, \; \bnu \in \Indx_0,
\]
we can write the vector $\mathbf{b}$ of samples as
\[
\mathbf{b} = \bPhi \mathbf{g}, \quad \mbox{ where } \mathbf{g} = (g_{\bnu})_{\bnu \in \Indx_0}.
\]
Since computing a sample $b_\ell = G(u(\bsy_\ell))$ involves the costly task of numerically solving a PDE, we prefer to work with a minimal number $m$ of samples so that $m < N$ and the above system
becomes underdetermined. Therefore, we propose the use of compressive sensing methods for the reconstruction such as
$\ell_1$-minimization, greedy algorithms or iterative hard thresholding, say, 
see e.g.\ \cite{fora13}. 
Rigorous recovery bounds in compressive sensing are usually achieved
for random matrices \cite{fora13} which suggests to choose 
the sampling points $\bsy_1,\hdots,\bsy_m$ independently at random according to the orthogonalization
measure $\eta$ in \eqref{prod:Cheb:meas}. 
In fact, results in \cite{fora13,ra10,rawa12,ruve08,cata06} 
based on the so-called restricted isometry property \cite[Chapter 6]{fora13}, see also below, 
state that an (approximately) $s$-sparse vector $\mathbf{g}$ 
can then be (approximately) recovered
from $\mathbf{b} = \bPhi \mathbf{g}$ via $\ell_1$-minimization 
(and other algorithms) with high probability provided that
\begin{equation}\label{m:bound:classical}
m \geq C K^2 s \log^3(s) \log(N),
\end{equation}
where $K = \max_{\bnu \in \Indx_0} \|T_{\bnu}\|_\infty$. 
We refer to \cite[Theorem 1] {BBRS15_2095} for a statement 
adapted to the present context.
By the definition of tensor product of Chebyshev-polonomials,
$\|T_{\bnu}\|_\infty = 2^{\|\bnu\|_0/2}$, 
where $\|\bnu\|_0$ counts the number of nonzero entries of the index $\bnu \in \cF$,
and hence,
$K = \max_{\bnu \in \Indx_0} 2^{\|\bnu\|_0/2}$. 
If $\max_{\bnu  \in \Indx_0} \|\bnu\|_0 =d$, 
then \eqref{m:bound:classical} reads 
\begin{equation}\label{m:bound:exponential}
m \geq C 2^d s \log^3(s) \log(N)\;.
\end{equation}
For small $d$ we can indeed conclude that compressive sensing approaches approximately 
recover $F$ from a small number of samples.
However, it can often be expected that $\Indx_0$ contains 
indices $\bnu$ with a significant number of 
nonzero entries (corresponding to many
non-trivial factors in the tensor product \eqref{tensCheb}), 
so that the above estimate obeys exponential scaling in $d$. 
In other words, we face the curse of dimension. 

In order to avoid the exponential scaling in $d$ in \eqref{m:bound:exponential}, 
we propose passing to \emph{weighted sparsity} and \emph{weighted $\ell_1$-minimization}, 
as introduced in \cite{rawa13}.
This requires stronger conditions on the expansion coefficients than just plain $\ell_p$-summability, namely
\[
\sum_{\bnu \in \cF} |g_\bnu|^p \omega_{\bnu}^{2-p} < \infty
\]
for a weight sequence 
$ (\omega_\bnu)_{\bnu\in \cF} $ 
that satisfies $\omega_{\bnu} \geq \|T_{\bnu}\|_\infty = 2^{\|\bnu\|_0/2}$. 

A contribution of our paper that chould be of independent interest shows weighted summability 
under a strengthened version of condition \eqref{UEA}. 
Suppose that for some weight sequence $(v_j)_{j \geq 1}$ with $v_j \geq 1$ and 
suitable constants $0 < \tilde{r} \leq \tilde{R} < \infty$, 
we have, for some $0 < p \leq 1$,
\begin{equation}\label{weighted:UEA}
\sum_{j \geq 1} v_j^{(2-p)/p} |\psi_j(x)|  
\leq 
\min\left\{ \bar{a}(x) - \tilde{r}, \tilde{R} - \bar{a}(x) \right\} \quad \mbox{ for all } x \in D
\end{equation}
and additionally
\begin{equation}\label{weighted:lp:assump}
\sum_{j \geq 1} \|\psi_j\|_\infty^p v_j^{2-p}  < \infty,
\end{equation}
then the expansion coefficients $d_{\bnu} \in V$ in \eqref{Cheb:expand:sol} satisfy 
\begin{equation}\label{weighted:lp:bound}
\sum_{\bnu} \|d_{\bnu}\|_V^p \omega_{\bnu}^{2 - p} < \infty, \quad \mbox{ where } 
\omega_{\bnu} = 2^{\|\bnu\|_0/2} \prod_{j \geq 1} v_j^{\nu_j}. 
\end{equation}
Consequently, also the coefficients $g_{\bnu} = G(d_{\bnu})$ satisfy 
$\sum_{\bnu \in \cF} |g_{\bnu}|^p \omega_{\bnu}^{2-p}<\infty$. 
In particular, these
weights satisfy $\omega_\bnu \geq \|T_{\bnu}\|_\infty$ as desired. 

A weighted version of compressive sensing 
has recently been introduced in \cite{rawa13}, which exactly fits our needs.
For instance, Stechkin's bound extends to the weighted case, see Section~\ref{sec:CSW} and 
\eqref{weighted:Stechkin} below for details. 
Consequently, a finite weighted $\ell_p$-norm of $(g_{\bnu})_{\bnu \in \cF}$ with $p < 1$ 
as in \eqref{weighted:lp:bound}  
implies corresponding convergence rates for (weighted) sparse approximation. 
As recovery method one may use weighted $\ell_1$-minimization \eqref{weighted:l1} 
or a weighted version of iterative hard thresholding, for instance \cite{jo13,fekrra14}. 
If the sampling points $\bsy_1,\hdots,\bsy_m$ are chosen independently at random 
according to the orthogonalization measure $\eta$ in \eqref{prod:Cheb:meas} 
and the weights satisfy $\omega_{\bnu} \geq \|T_{\bnu}\|_\infty$ -- 
as valid for the choice in \eqref{weighted:lp:bound} -- then recovery
guarantees for weighted $\ell_1$-minimization were shown in \cite{rawa13}. 
These require that the number $m$ of samples satisfies
\[
m \geq C s \log^3(s) \log(N).
\]
In contrast to \eqref{m:bound:exponential},
this bound is free from the curse of dimensionality
in the sense that it does not scale with $2^d$ anymore. 
Based on these ingredients, 
we introduce an algorithm (see Section~\ref{sec:CS}) that numerically 
computes an approximation of $F(\bsy) = G(u(\bsy))$ using 
(a relatively small number $m$ of) sample evaluations
$F(\bsy_\ell)$, $\ell=1,\hdots,m$.
We show that the approximation $F^\sharp$ computed by our algorithm satisfies
\begin{align}
\| F - F^\sharp\|_{L_2(U,\eta)} & \leq C \left(\frac{ \log^3(m)\log(N)}{m} \right)^{1/p-1/2}, 
\label{L2bound}
\\
\| F - F^\sharp\|_{L_\infty(U,\eta)} & \leq C \left( \frac{\log^3(m)\log(N)}{m} \right)^{1/p-1} 
\label{Linftybound}
\end{align}
under the conditions \eqref{weighted:UEA} and \eqref{weighted:lp:assump}, 
where $N = \#\Indx_0$ is the cardinality
of the initial finite index set (depending on the desired sparsity level $s$). 
The constant $C$ in \eqref{L2bound}, \eqref{Linftybound}
only depends on $F$ and $p$ through the weighted 
$\ell_p$-norm of the expansion coefficients $(g_{\bnu})$.
Comparing to the rate \eqref{best:approx:rate} for the best $s$-term approximation reveals that \eqref{L2bound} and \eqref{Linftybound}
are optimal up to logarithmic factors.
It may be surprising at first sight that we also obtain an $L_\infty$-bound, 
which for other methods is usually obtained only under additional
regularity assumptions on the functions $\psi_j$.

In the important case that the weights grow polynomially, i.e., $v_j = \beta j^\alpha$ for some 
$\beta > 1, \alpha > 0$, then our proposed choice of $\Indx_0= \Indx_0^s$ satisfies $N =\# \Indx_0^s \leq c_{\alpha,\beta} s^{d_{\alpha,\beta} \log(s)}$ and \eqref{L2bound}, \eqref{Linftybound} 
imply then
\[
\| F - F^\sharp\|_{L_2(U,\eta)} \leq C_{\alpha,\beta} \left( \frac{\log^5(m)}{m} \right)^{1/p-1/2}, \qquad 
\| F - F^\sharp\|_{L_\infty(U,\eta)} \leq C_{\alpha,\beta} \left( \frac{\log^5(m)}{m} \right)^{1/p-1}.
\]
Comparing to the best $s$-term approximation rate \eqref{best:approx:rate}, we achieve an optimal convergence rate up to logarithmic factors, 
using $m$ sample evaluations. 
Polynomially growing weights $v_j= \beta j^\alpha$ may appear when the norms of the functions $\psi_j$ in \eqref{eq:aKL}
(arising for instance via a Karhunen-Lo{\'e}ve expansion of a random field modeling the diffusion $a$)
decay polynomially like $\|\psi_j\|_\infty \leq c j^{-\tau}$ for some $\tau > \alpha + 1$.
We refer to Theorem~\ref{thm:main:func}, Remark~\ref{rem:diff:eq}, Remark~\ref{remk:ChoiWgt}, Corollary~\ref{def:Indxs} and Remark~\ref{rem:Indx:size}
for precise statements.

Our algorithm requires that the sample solutions are computed up to a certain accuracy. 
We propose to use a Petrov-Galerkin (PG) approach for this task.
As usual, convergence rate estimates for the PG discretization errors
require additional smoothness.

In the remainder of the paper, we work in an abstract setting which includes the case of the parametric diffusion equation, but applies
to a significantly larger range of problems. The reader may keep in mind the specific setup outlined in this introduction as a guiding example.

The outline of this paper is as follows.
In Section \ref{sec:AffParOpEqn} we introduce the class of 
affine-parametric operator equations which will be considered
in the sequel and provide background on previous results.
Section~\ref{sec:CSW} is devoted to the background on compressive sensing and in particular
to its recent extension in \cite{rawa13} to weighted sparsity. Section~\ref{sec:WeighEllp} provides the new weighted $\ell_p$-bounds
on the solution coefficients in the polynomial chaos expansion in terms of tensorized \Tscheb polynomials. 
The compressive sensing algorithm for approximating the functional applied to the solution of an affine parametric operator equation
is introduced and analyzed in Section~\ref{sec:CS}.

\section{Background on Affine-Parametric Operator Equations}
\label{sec:AffParOpEqn}
Generalizing the parametric diffusion equation from above 
we introduce a class of affine-parametric operator equations
which depends, at least formally, on a possibly countable set of 
parameters. We aim at numerical methods for the approximate, numerical
solution of these equations. To this end, we propose a combined 
strategy which is based on (Petrov-)Galerkin projection in the physical
variables, and compressive sensing techniques in the high-dimensional
parameter domain. Approaches of this type have attracted interested
in recent years in the area of computational uncertainty quantification
for partial differential equations \cite{doow11,hado14,jaelsa14,pehado14,sasanaderith14,yaka13}.
In this section, we will provide the necessary background information concerning
affine-parametric operator equations and Petrov-Galerkin projections.

\subsection{Affine-Parametric operator equations}
\label{ssec:affparops}
Consider $\bcX$ and $\bcY$, 
two separable and reflexive Banach spaces.
Unless explicitly stated otherwise, we assume 
$\bcX$ and $\bcY$ to have coefficient field $\mathbb{R}$.
The (topological) duals are denoted by $\bcX'$ and $\bcY'$, respectively. 
As usual, $\cL(\bcX,\bcY')$ is the set of bounded linear operators 
$A:\bcX \to\bcY'$.

We 
aim at solving, for given $f\in \bcY'$ and for 
a parameter sequence $\bsy\in U = [-1,1]^\N$,
the {\em parametric operator equation}
\begin{equation}\label{eq:main}
  A(\bsy)\, u (\bsy) = f,
\end{equation}
where we assume that $A(\bsy)$ depends in an affine manner on $\bsy$, i.e.,
\begin{equation}\label{eq:Baffine}
 A(\bsy) = A_0 + \sum_{j\ge 1} y_j\, A_j 
\end{equation}
with a sequence $\{ A_j \}_{j\geq 0}\subset \cL(\bcX,\bcY')$ 
which we assume to be summable in the sense that
$\sum_{j\geq 1} \| A_j \|_{ \cL(\bcX,\bcY')} < \infty$.
We shall in particular do this by {\em collocation} of \eqref{eq:main},
i.e., by (approximately) solving \eqref{eq:main} for particular instances
of the parameter sequence $\bsy$; hence, conditions are needed which ensure
that the sum in \eqref{eq:Baffine} converges {\em pointwise}, 
i.e., for every $\bsy\in U$, in the sense of $\cL(\bcX,\bcY')$.

We associate with the operators $A_j$ the bilinear
forms $\fa_j(\cdot,\cdot):\bcX\times \bcY \rightarrow \mathbb{R}$ via
$$
  \fa_j(v,w) \,=\, {_{\cY'}}\langle  A_j v, w \rangle_{\cY} \quad \mbox{ for } v\in \cX,\;w\in \cY,
  \quad j=0,1,2,\ldots \;,
$$
where  ${_{\cY'}}\langle  u,w \rangle_{\cY} = u(w)$ for $w \in \cY$ and $u \in \cY'$ as usual.
Similarly, for every parameter instance $\bsy\in U$, 
we associate with $A(\bsy)$ the 
parametric bilinear form 
$\fa(\bsy;\cdot,\cdot): \bcX\times\bcY\to\R$ 
via
\begin{equation}\label{eq:parmbil}
   \fa(\bsy;v,w) \,=\, {_{\cY'}}\langle A(\bsy) v, w\rangle_{\cY} \quad  \mbox{ for } v\in \cX, w\in \cY.
\end{equation}
We work under the following conditions 
on $\{A_j\}_{j\geq 0}$ 
(cp.\ \cite{Kuo475}).
\begin{assume}\label{ass:AssBj}
\begin{enumerate}
\item%
$A_0\in \cL(\bcX,\bcY')$ is boundedly invertible: 
there exists $\mu_0 > 0$ s.t.
\begin{equation}\label{eq:B0infsup} 
 \inf_{v \in \bcX \setminus \{0\}} \sup_{w \in \bcY \setminus \{0\}}
 \frac{\fa_0(v,w)}{\| v \|_{\bcX} \|w\|_{\bcY}}
 \ge \mu_0,\quad
 \inf_{w \in \bcY \setminus \{0\}}\sup_{v \in \bcX \setminus\{0\}}
 \frac{\fa_0(v,w)}{\| v \|_{\bcX} \|w\|_{\bcY}}
 \ge \mu_0.
\end{equation}
\item%
there exists a constant $0 < \kappa < 1$ such that
\begin{equation} \label{eq:Bjsmall} 
 \sum_{j\geq 1} \beta_{0,j} \leq \kappa \;,
 \quad\mbox{where}\quad
 \beta_{0,j} \,:=\, \| A_0^{-1} A_j \|_{\cL(\cX,\cX)}\;,
 \quad j=1,2,\ldots
 \;.
\end{equation}
\end{enumerate}
\end{assume}
The parametric diffusion equation \eqref{eq:Diffusion} falls into this setup 
with $A_0 u:= -\nabla \cdot (\bar{a} \nabla u)$ 
and $A_j u := - \nabla \cdot (\psi_j \nabla u)$.
A symmetric variational form results with
the choice 
$\bcX = \bcY = H_0^1(D)$ and $\bcX' = \bcY' = H^{-1}(D)$. 
Note that
\[
\beta_{0,j} = \|A_0^{-1} A_j \|_{\cL(\cX,\cX)} 
\leq  
\| A_j \|_{\cL(\cX,\cX')}  \|A_0^{-1}\|_{\cL(\cX',\cX)}
\leq \|\psi_j\|_\infty (\inf_{x \in D} |\bar{a}(x)|)^{-1},
\]
so that \eqref{eq:Bjsmall} is implied by
\be\label{eq:Ksmall}
\sum_{j\geq 1} \|\psi_j\|_\infty \leq \kappa \inf_{x \in D} |\bar{a}(x)|\;.
\ee
This condition is slightly stronger than the uniform ellipticity assumption \eqref{UEA}. 
This slight drawback of the described general setup
is compensated by the fact that the results to be developed will hold
for operator equations beyond the model parametric diffusion equation
\eqref{eq:Diffusion} such as linear elasticity,
Helmholtz and Maxwell equations with affine-parametric operators;
we refer to \cite{HS13_492} for a class of 
parametric, elliptic multiscale problems, and to 
\cite{Hanse1085} for affine-parametric, nonlinear initial value problems.
The possibility of choosing $\cX \ne \cY$ accommodates
saddle point formulations of parabolic evolution equations.
\begin{example}\label{expl:Parabolic}
In a bounded domain $D\subset \R^n$, and in the time interval
$I=(0,T)$ for a time-horizon $0<T<\infty$, 
and for the affine-parametric, elliptic operator 
$A(\bsy)$ in \eqref{eq:Diffusion},
for given $f(x,t)$ and for given $u_0\in L_2(D)$, we
consider the parametric, linear parabolic evolution problem
\be\label{eq:ParEvol}
B(\bsy)u := \partial_t u -  A(\bsy) u = f \;,
\quad 
u(\cdot,t)|_{\partial D} = 0 
\quad \mbox{in}\quad (0,T)\times D\;,
\quad 
u(\cdot,0) = u_0 \;.
\ee
We assume that $A(\bsy)$ is given by an expansion as in \eqref{eq:Baffine}. Then, 
the parabolic, parametric evolution operator $B(\bsy)$ in \eqref{eq:ParEvol}
was shown in \cite[Sec. 5]{SS09_395} to fit into the 
abstract Assumption~\ref{ass:AssBj} 
with 
$\bcX = L_2(I;V)\cap H^1(I;V')$, $\bcY = L_2(I;V)\times H$, 
$V=H^1_0(D)$ and $H=L_2(D)$. Here, the parametric bilinear form
$ \fa(\bsy;w,v) $ is defined, for $v=(v_1,v_2)\in L_2(I;V)\times H$,
by
$$
 \fa(\bsy;w,v)
 : = 
\int_I 
\langle{ \textstyle \frac{d w}{dt}}(t), v_1(t)\rangle_H 
+ 
\int_D a(x,\bsy) \nabla w\cdot \nabla v_1 dx
dt
+ 
\langle w(0),v_2\rangle_H
\;,
$$
where $a(x,\bsy)$ denotes the affine-parametric, isotropic diffusion
coefficient in \eqref{eq:aKL}. 
Then, under Assumption~\eqref{UEA} Condition (i) 
in Assumption~\ref{ass:AssBj} holds true; we refer to \cite[Appendix]{SS09_395} for a proof.
Condition (ii) in Assumption \ref{ass:AssBj} is implied by \eqref{eq:Ksmall},
for sufficiently small $\kappa$.
\end{example}
The next result states existence and boundedness of solutions $u(\bsy)$.
\begin{prop}[{cp.~\cite[Theorem 2]{ScMCQMC12}}] \label{prop:BsigmaInv} 
Under Assumption~\ref{ass:AssBj}, for every
realisation $\bsy\in U$ of the parameter vector, the affine parametric
operator $A(\bsy)$ given by \eqref{eq:Baffine} is boundedly invertible.
In particular, for every $f \in \bcY'$ and for every $\bsy \in U$, 
the parametric operator equation
\begin{equation} \label{eq:parmOpEq}
 \mbox{find} \quad u(\bsy) \in \bcX:\quad
 \fa(\bsy;u(\bsy), w) \,=\,  {_{\bcY'}}\langle  f,w \rangle_{\bcY}
 \quad
 \mbox{ for all } w \in \bcY
\end{equation}
admits a unique solution $u(\bsy)$ which satisfies with  
$\mu = (1 - \kappa)\mu_0$ the uniform a-priori estimate
\begin{equation}\label{eq:apriori}
\sup_{\bsy\in U} \| u(\bsy) \|_{\bcX}
 \,\leq\,
 \frac{1}{\mu}
 \, \| f \|_{\bcY'} 
\;.
\end{equation}
\end{prop}
%

\subsection{Parametric and spatial regularity of solutions}
\label{ssec:anadepsol}

To obtain convergence rates of (Petrov-)Galerkin discretizations,
we introduce 
{\em scales of smoothness spaces} 
$\{ \cX_t \}_{t \geq 0}$, $\{ \cY_t \}_{t\geq 0}$, 
$\{ \cX'_t \}_{t \geq 0}$, $\{ \cY'_t \}_{t\geq 0}$
with
\begin{equation}\label{eq:SmScal}
\begin{aligned}
 \cX &= \cX_0 \supset \cX_1 \supset \cX_2 \supset \cdots\;,
 &\cY &= \cY_0 \supset \cY_1 \supset \cY_2 \supset \cdots\;,
 \quad\mbox{and}
 \\
 \cX' &= \cX'_0 \supset \cX'_1 \supset \cX'_2 \supset \cdots\;,
 &\cY' &= \cY'_0 \supset \cY'_1 \supset \cY'_2 \supset \cdots
 \;.
\end{aligned}
\end{equation}
For noninteger values of $t$, we define the scales 
$\{ \cX_t \}_{t\geq 0}$ and $\{ \cY'_t \}_{t\geq 0}$
by interpolation.
In the particular case of scalar, elliptic self-adjoint operators, 
$\cX_t  = \cY_t$.
We formalize the parametric regularity hypothesis below.
\begin{assume}\label{ass:XtYt}
There exists $\bar{t} > 0$ such that the following conditions hold:
\begin{enumerate}
\item 
For every $t$ satisfying $0 < t \le\bar{t}$, 
we have uniform parametric regularity
\begin{equation}\label{eq:Regul}
 \sup_{\bsy\in U} \| A(\bsy)^{-1} \|_{\cL(\cY'_t, \cX_t)} < \infty
 \quad\mbox{and}\quad
 \sup_{\bsy\in U} \| (A^*(\bsy))^{-1} \|_{\cL(\cX'_{t}, \cY_{t})}  
< 
\infty.
\end{equation}
\item
For $0< t \leq \bar{t}$ there exists $0<p_t< 1$ such that
\begin{equation} \label{eq:assump_t}
 \sum_{j \ge 1} \| A_j \|^{p_t}_{\cL(\cX_t,\cY'_t)} < \infty,
\end{equation}
and $0< p_0 \le p_t \le p_{\bar{t}}<1$.
\item 
We assume that the operators $A_j$ are enumerated such that the sequence
$\bbeta_0$ in \eqref{eq:Bjsmall} satisfies
\begin{equation} \label{eq:ordered}
  \beta_{0,1} \ge \beta_{0,2} \ge \cdots \ge \beta_{0,j} \ge \, \cdots .
\end{equation}
\end{enumerate}
\end{assume}

\medskip

\begin{remark}\label{rmk:AinvNeum}
Under condition \eqref{eq:B0infsup}, $A_0 : \bcX \to \bcY'$ is boundedly invertible,
so that for every $\bsy\in U$ we may write
$A(\bsy) = A_0(I + \sum_{j\geq 1}y_j A_0^{-1}A_j)$.
Then, for \eqref{eq:Regul} to hold it is sufficient that 
$A_0^{-1}\in \cL(\bcY'_t,\bcX_t)$, $A_j \in \cL(\bcX_t,\bcY'_t)$
and that
\be\label{eq:betatsmall}
  \sum_{j\ge1} \beta_{t,j}  < 1\;,\quad 
\mbox{where} \;\; \beta_{t,j} := \| A_0^{-1}A_j \|_{\cL(\bcX_t,\bcX_t)} 
\;.
\ee
Furthermore, \eqref{eq:assump_t} is equivalent to
\[
\| \bbeta_t \|_{\ell^{p_t}}^{p_t} = \sum_{j\geq 1} \beta_{t,j}^{p_t} < \infty, 
\]
since, for every $j\geq 1$,
\begin{equation}\label{eq:Aknorm}
  \| A_0 \|_{\cL(\cX_t,\cY_t')}^{-1} 
  \le \frac{\| A_0^{-1} A_j \|_{\cL(\bcX_t,\bcX_t)}}{ \| A_j \|_{\cL(\bcX_t,\bcY'_t)} }
  \le \| A_0^{-1} \|_{\cL(\bcY'_t, \bcX_t)}
\;.
\end{equation}
\end{remark}
\begin{remark}\label{rmk:ParReg}
Assumption \ref{ass:XtYt}(i) 
on the uniform (w.r.t. $\bsy\in U$) 
boundedness of the inverse $(A(\bsy))^{-1}$ has the following 
uniform regularity implication. 
For $0 < t \leq \bar{t}$, $\bsy\in U$, $f\in \cY'_t$ and $H\in\cX'_{t}$
define $u(\bsy) = (A(\bsy))^{-1}f$ and $w(\bsy) = (A^*(\bsy))^{-1}H$. 
Then
there exist constants $C_t,C'_{t}>0$, independent of $f$ and $H$, 
such that
\begin{equation}\label{ass:a priory}
 \sup_{\bsy\in U} \|u(\bsy)\|_{\cX_t} \le C_t \|f\|_{\cY'_t}
 \quad\mbox{and}\quad
 \sup_{\bsy\in U} \|w(\bsy)\|_{\cY_{t}} \le C'_{t} \|H\|_{\cX'_{t}}
\;.
\end{equation}
\end{remark}
In \cite{NS13_782}, for the  smoothness scales 
$\{ \bcX_t \}_{t\geq 0}$ and $\{ \bcY'_t\}_{t\geq 0}$ 
being weighted Sobolev spaces or Besov spaces
in certain bounded Lipschitz polyhedra $D\subset \R^3$,
the regularity \eqref{ass:a priory} was established
for second order parametric, elliptic systems. 
For these spaces, 
simplicial finite elements in $D$ on meshes with proper refinement
towards corners and edges of the Lipschitz polyhedron $D$ 
will yield best possible convergence rates 
for $0<t\leq \bar{t}$ where $\bar{t}$ depends on the regularity
of the data and on the degree of the elements.

\subsection{Petrov-Galerkin discretization}
\label{ssec:GalDisc}
For a given parameter value $\bsy\in U$ 
we solve the operator equation \eqref{eq:main} 
approximately by {\em Petrov-Galerkin discretization} 
(``PG discretization'' for short). 
This will be one building block
of the compressive-sensing Petrov-Galerkin discretization 
of \eqref{eq:main} to be developed below.

Specifically, 
we consider $\{ \cX^h \}_{h>0}\subset\cX$ and $\{ \cY^h \}_{h>0}\subset\cY$,
two one-parameter families of nested,
finite dimensional subspaces which are 
dense in $\cX$ and in $\cY$, respectively, 
as $h\downarrow 0$. 
We assume that the 
subspace families $\{ \cX^h \}_{h>0}\subset\cX$
and $\{ \cY^h \}_{h>0}\subset\cY$
have the {\em approximation properties}:
for $0 < t \leq \bar{t}$, 
there exist constants $C_t,C'_{t} > 0$ such that
for all
$0< h \leq 1$, all $u \in \cX_t$ and all $w \in \cY_{t}$ 
holds
\begin{equation}\label{eq:apprprop}
\begin{array}{rcl}
\displaystyle
\inf_{u^h\in \cX^h} \| u - u^h \|_{\cX}
& \leq &
C_t h^t \| u \|_{\cX_t},\\
\displaystyle
\inf_{w^h\in \cY^h} \| w - w^h \|_{\cY}
& \leq &
C_{t}' h^{t} \| w \|_{\cY_{t}} 
\;.
\end{array}
\end{equation}
We will also assume that
the subspace sequences $\{ \cX^h \}_{h>0}\subset\cX$ and 
$\{ \cY^h \}_{h>0}\subset\cY$ are 
{\em uniformly inf-sup stable}.
This is to say that there exist $\bar{\mu} > 0$
and $h_0 > 0$ such that for every $0<h \leq h_0$, 
the uniform 
discrete inf-sup conditions
\begin{align}
\label{eq:Bhinfsup1}
\inf_{v^h \in \bcX^h \setminus \{0\}} \sup_{w^h \in \bcY^h \setminus \{0\}}
\frac{\fa(\bsy;v^h,w^h)}{\| v^h \|_{\bcX} \|w^h\|_{\bcY}}
& \geq \bar{\mu}  > 0 \quad \mbox{ for all } \bsy \in U, \\
\label{eq:Bhinfsup2}
\inf_{w^h \in \bcY^h \setminus \{0\} } \sup_{v^h \in \bcX^h \setminus \{0\}}
\frac{\fa(\bsy;v^h,w^h)}{\| v^h \|_{\bcX} \|w^h\|_{\bcY}}
& \geq \bar{\mu}>0 \quad \mbox{ for all } \bsy \in U
\end{align}
hold.
\begin{prop}\label{prop:Galerkin}
Assume the \emph{approximation property}  \eqref{eq:apprprop},
the uniform stability \eqref{eq:Bhinfsup1}, \eqref{eq:Bhinfsup2},
and uniform parametric regularity \eqref{eq:Regul}
in Assumption~\ref{ass:XtYt}(i) hold.
This implies existence, uniqueness and 
asymptotic \emph{quasioptimality}
of Petrov-Galerkin approximations: 
there exists $h_0 > 0$ such that
for every $0 < h \leq h_0$ and for every $\bsy \in U$, 
the Galerkin approximations $u^h(\bsy)\in\cX^h$, given by
\begin{equation} \label{eq:parmOpEqh}
\mbox{find} \; u^h(\bsy) \in \bcX^h \mbox{ satisfying }
\quad
\fa(\bsy;u^h(\bsy),w^h) 
= 
{_{\cY'}} \langle f, w^h \rangle_{\cY}
\quad
\mbox{ for all } w^h\in \bcY^h,
\end{equation}
are well defined, and stable, i.e., they satisfy the 
uniform a-priori estimate
\begin{equation}\label{eq:FEstab}
 \| u^h(\bsy) \|_{\bcX} \,\le\, \frac{1}{\bar{\mu}}\, \| f \|_{\bcY'} \quad 
\mbox{ for all } \bsy \in U.
\end{equation}
Moreover, for $0<t\le \bar{t}$ there exists a constant $C_t>0$ such that
\begin{equation} \label{eq:reg1}
 \| u(\bsy) - u^h(\bsy) \|_{\cX} \,\le\, C_t\, h^t\, \| u(\bsy)\|_{\cX_t} 
                \quad \mbox{ for all } \bsy \in U.
\end{equation}
\end{prop}
\begin{remark}\label{remk:UniDisInfSup}
Under Assumption~\ref{ass:AssBj}, items (i) and (iii),
the validity of the discrete inf-sup
conditions for the nominal bilinear form $\fa_0(\cdot,\cdot)$,
see~\eqref{eq:B0infsup}, with constant
$\bar{\mu}_0>0$ independent of $h$, 
implies the uniform discrete inf-sup conditions for the bilinear form
$\fa(\bsy;\cdot,\cdot)$ with $\bar{\mu} = (1-\kappa) \bar{\mu}_0 > 0$,
as assumed in Proposition \ref{prop:Galerkin}.
\end{remark}

For proving superconvergence of
continuous, linear functionals $G \in \cX'_{t}$,
we also assume uniform inf-sup stability of the pairs
$\bcX^h\times \bcY^h$ for the adjoint problem. 
This assumption implies that for every 
$0 < t \le \bar{t}$ there exists a constant $C_t>0$ 
such that, for all $0<h\le h_0$ and all $\bsy\in U$,
\begin{equation} \label{eq:reg2}
 \| w(\bsy) - w^h(\bsy) \|_{\cY} \,\le\, C_{t} h^{t} \| w(\bsy)\|_{\cY_{t}}.
\end{equation}
\begin{prop}[{cp.~\cite[Theorems 2.4 and 2.5]{DKLN13_1134}}]\label{prop:FEconvrate}
Under Assumption~\ref{ass:AssBj} and Condition \eqref{eq:Regul}, for every
$f\in \cY'$ and for every $\bsy\in U$, the approximations $u^h(\bsy)$
are stable, i.e., \eqref{eq:FEstab} holds.
For every $f\in \cY'_t$ with $0<t\le \bar{t}$, 
there exists a constant $C>0$ such that as 
$h\rightarrow 0$ there holds
\begin{equation} \label{eq:FEconvu}
  \sup_{\bsy} \| u(\bsy) - u^h(\bsy) \|_\cX \,\le\, C\, h^t\, \| f \|_{\cY'_t }.
\end{equation}
Moreover, for a functional $G \in \cX'_{t'}$ with $0<t'\leq \bar{t}$,
for every $f\in \cY'_{t}$ with $0<t\leq \bar{t}$, for every $\bsy\in U$, 
as $h\to 0$, there exists a constant $C>0$ independent of 
$h>0$ and of $\bsy\in U$
such that the Galerkin approximations $G(u^h(\bsy))$ satisfy
the asymptotic error bound
\begin{align} \label{eq:Gconvest}
  \left| G(u(\bsy)) - G(u^h(\bsy)) \right|
  &\,\le\, C\, h^{t+t'} \, \| f \|_{\cY'_{t}} \, \| G \|_{\cX'_{t'}}.
\end{align}
\end{prop}
%
\subsection{Dimension truncation}
\label{sec:dimtrunc}
It will be convenient below to truncate the infinite sum in \eqref{eq:Baffine} to $B$ 
terms and 
solve the corresponding operator equation \eqref{eq:main} 
approximately using Galerkin discretization from 
two dense, one-parameter families 
$\{\cX^h\}\subset \cX$, $\{\cY^h\}\subset \cY$ of subspaces
of $\cX$ and $\cY$.
For $J\in \N$ and $\bsy\in U$, we define
\begin{equation}\label{equ:defAs}
\fa_{J}(\bsy;v,w) := \Ylangle A^{(J)}(\bsy) v, w\Yrangle,
\quad\text{with}\quad
  A^{(J)}(\bsy) :=  A_0 + \sum_{j=1}^{J} y_j A_j\;.
\end{equation}
Then, for every $0<h \le h_0$ and every $\bsy\in U$, 
the dimension-truncated Galerkin solution $u^h_{J}(\bsy)$ is the solution of
\begin{equation} \label{eq:uhs}
\text{find } u^h_{J}(\bsy) \in \cX^h: \quad
\fa_J(\bsy;u^h_{J}(\bsy),w^h) = \Ylangle f, w^h \Yrangle \quad
\forall w^h\in \cY^h\;.
\end{equation}
By choosing $\bsy = (y_1,\ldots,y_{J},0,0,\ldots)$,
Proposition~\ref{prop:Galerkin} and Remark~\ref{remk:UniDisInfSup}
remain valid for the dimensionally truncated
problem \eqref{eq:uhs}, and hence \eqref{eq:FEstab} 
holds with $u^h_{J}(\bsy)$ in place of $u^h(\bsy)$, that is
\begin{equation}\label{eq:FEstab_s}
\sup_{\bsy}  \|u^h_{J}(\bsy)\|_{\cX} \le \frac{1}{\bar{\mu}} \|f\|_{\cY'}
\;.
\end{equation}

\begin{thm} \label{thm:trunc} (\cite[Theorem~5.1]{Kuo475})
Under Assumption~\ref{ass:AssBj}, 
for every $f\in \cY'$, $G\in \cX'$, 
$\bsy\in U$, $J\in\N$ and $h>0$,
the variational problem \eqref{eq:uhs} 
admits a unique solution $u_{J}^h(\bsy)$ 
which satisfies
\begin{equation}\label{eq:Idimtrunc}
  |G(u^h)- G(u^h_{J})|
  \,\le\, C\, 
  \|f\|_{\cY'}\, \|G\|_{\cX'}\,
  \bigg(\sum_{j\ge J +1} \beta_{0,j}\bigg)^2
\end{equation}
for some constant $C>0$ 
independent of $f$, $G$ and of $J$ 
where $\beta_{0,j}$ is defined in \eqref{eq:Bjsmall}. 
In addition, if \eqref{eq:ordered} and Assumption \ref{ass:AssBj} 
hold with some $0< p_0 < 1$, 
then
\begin{equation}\label{trunc:Bp}
  \sum_{j\ge J + 1} \beta_{0,j}
  \,\le\,
 \min\left(\frac{1}{1/p_0-1},1\right)
  \bigg(\sum_{j\ge1} \beta_{0,j}^{p_0} \bigg)^{1/p_0}
  J^{-(1/{p_0}-1)}\;.
\end{equation}
Analogous bounds, with constants which do not depend on the discretization
parameter $h$ also hold for the parametric PG solutions $u^h(\bsy)$
in \eqref{eq:parmOpEqh}.
\end{thm}
\subsection{Tensor Product \Tsch Expansion}
\label{sec:Tscheb}
It has been shown in several contributions that the solution
$u(\bsy)$ can be represented via polynomial chaos expansions in terms of 
multivariate monomials \cite{codesc10-1,codesc11}, 
tensorized Legendre polynomials \cite{chalch14,codesc10-1,codesc11} 
and tensorized \Tscheb polynomials \cite{Hanse486}.
The coefficient sequence in these expansions
is contained in an $\ell_p$-space with $0 < p \leq 1$. 
This fact enables to show sparse approximation rates, in particular,
the solution can be well approximated via a finitely
truncated polynomial expansion with only a few terms. 
We will extend such estimates below to the weighted case.

To precise our notion of sparsity,
as in \cite{Hanse486}
we consider an expansion of the parametric solution
$u = u(\bsy) \in \bcX$ into formal series
with respect to tensorized \Tscheb polynomials, i.e.,
\be\label{Cheb:expand}
u(\bsy) = \sum_{\bnu \in \cF} d_{\bnu} T_{\bnu}(\bsy).
\ee
Note that the coefficients $d_{\bnu}$ are elements of $\bcX$, and in the specific case of the parametric diffusion equation they are functions in
$V = H_0^1(D)$, 
The convergence of the (formal, at this stage)
infinite sum in \eqref{Cheb:expand} is, in general,
unconditional (see, eg., \cite{codesc10-1,codesc11} for 
analogous assertions for tensorized Taylor- and 
Legendre polynomials, and \cite{Hanse486} for
\Tscheb polynomials).
Also, as developed in \cite{codesc10-1,codesc11}
initially for Legendre polynomials,
it has been shown in \cite{Hanse486} (for a more general version of the parametric diffusion equation) that 
the sequence $(\| d_{\bnu} \|_V)_{\bnu \in \cF}$ 
is contained in $\ell^p(\cF)$ for $0 < p \leq 1$ provided that
$(\|\psi_j\|_\infty)_{j \in \N}$ is contained in $\ell^p(\N)$. 
We will generalize this sparsity result for \Tscheb
coefficients to the weighted case 
(and in the general abstract context of affine parametric operator equations). 
As a consequence of a well-known estimate 
(going back to Stechkin, see also \cite[Theorems 2.3 and 2.5]{fora13})
it follows that $u$ can be well-approximated by
finitely truncated expansions \eqref{Cheb:expand},
that is, by
\[
u_{\Lambda}(\bsy) = \sum_{\bnu \in \Lambda} d_{\bnu} T_{\bnu}(\bsy),
\]
where $\Lambda \subset \cF$ has small cardinality, say $\#\Lambda =s$. 
We refer to \cite{codesc10-1,codesc11,Hanse486} for details.

This observation is the motivation for our method based 
on compressive sensing for numerically approximating the 
solution $u = u(\bsy)$, or at least functionals 
$G: \bcX \to \R$ of the parametric solution $F(\bsy) = G (u(\bsy))$ 
described and analyzed in this paper.
We refer to Remark~\ref{rem:other:pol} below for reasons why we have chosen
$\Tscheb$ polynomials in our context.

\section{Compressive Sensing}
\label{sec:CSW}
Compressive sensing allows to solve underdetermined systems of linear equations 
under certain conditions.
Our goal is to apply this principle for recovering an approximation 
to the parametric solution $u(\bsy)$ of \eqref{diff:eq:weak} from (approximate) samples $u(\bsy_\ell)$, $\ell = 1,\hdots,m$.
For the sake of simplicity, however, we consider here the recovery of a functional 
$G(u(\bsy))$ of the parametric solution from samples $G(u(\bsy_\ell))$, and postpone the reconstruction of the full solution $u(\bsy)$
to a later contribution.
Since computing a sample corresponds to a numerical solution of a 
PDE, which is usually a costly procedure, we would like to take
only a small number $m$ of samples. Before going into details on this, 
we first review basic results from compressive sensing.

Given a matrix $\bsA \in \R^{m \times N}$ with $m \ll N$ the general 
aim of compressive sensing is to reconstruct a vector $\bsx \in \R^N$ 
from
\[
\bsb = \bsA \bsx.
\]
(Later $\bsx$ will be replaced by $(g_\bnu) = (G(d_{\bnu}))$ for the coefficients
$d_{\bnu}$ in \eqref{Cheb:expand}.)
Since this is an underdetermined linear system of equations, 
additional assumptions are required in order to be able to (approximately) reconstruct
$\bsx$. Here we work with the assumption that $\bsx$ is $s$-sparse, 
i.e., $\|\bsx\|_0 = \#\{ \ell: x_\ell \neq 0\} \leq s$ 
for sufficiently small $s < m$ or that at least the error of best $s$-term approximation
\[
\sigma_s(\bsx)_1 := \inf_{\bsz: \|\bsz\|_0 \leq s} \|\bsx - \bsz \|_1
\]
is small (decays quickly in $s$). 
The na{\"i}ve approach of recovering $\bsx$ from $\bsb = \bsA\bsx$ via $\ell_0$-minimization 
is unfortunately NP hard \cite{fora13},
but several tractable alternatives have been developed.
Most notably among these is $\ell_1$-minimization 
which consists in finding the minimizer of
\[
\min \|\bsz\|_1 \quad \mbox{ subject to } \bsA \bsz = \bsb\;.
\]
For a number of random matrix constructions of $\bsA$, 
it can be shown that recovery of $s$-sparse vectors via 
$\ell_1$-minimization (and other algorithms) 
is successful and stable with high probability provided that 
$m \asymp s \log^\alpha(N)$ 
(where $\alpha = 1$ or $\alpha =4$ for certain constructions).

A particular setup of interest to us arises from random sampling of sparse expansions:
consider functions on a set $U$ of the form
\be\label{eq:fx}
u(\bsy) = \sum_{j \in \Indx} x_j \phi_j(\bsy), \quad \bsy \in U,
\ee
where the $\phi_j:U \to \R$, $j \in \Indx$, form a finite orthonormal system
(later we extend to infinite orthonormal systems)
with respect to a probability measure $\eta$ on $U$, i.e.,
\begin{equation}\label{orthogonality}
\int_U \phi_j(t) \phi_k(t) d\eta(t) = \delta_{jk}, \quad j,k \in \Indx.
\end{equation}
(Below, we will choose $U= [-1,1]^\N$, $\eta$ to be the product probability measure \eqref{prod:Cheb:meas}, the set $\Indx$ to be $\cF$ defined in \eqref{def:cF}
and the $\phi_j$ to be the tensorized \Tscheb polynomials $T_{\bnu}$.)
The function $u$ is called $s$-sparse if the coefficient vector $\bsx$ 
in \eqref{eq:fx} is $s$-sparse. 
Given samples $u(\bsy_1),\hdots,u(\bsy_m)$ at locations $\bsy_\ell \in U$ we
would like to reconstruct $u$. Introducing the sampling matrix
\begin{equation}\label{def:sampling:matrix}
\A_{\ell,j} = \phi_j(\bsy_\ell), \quad \ell=1,\hdots,m, \; j \in \Indx,
\end{equation}
the vector $\bsb$ of samples $b_\ell = u(\bsy_\ell)$ can be written as 
\[
\bsb = \bsA \bsx,
\]
and for (approximately) sparse $\bsx$ and small $m \ll N$ 
we obtain a compressive sensing problem with this particular type of matrix.
If the samples $\bsy_1,\hdots,\bsy_m$ are chosen independently at random according 
to the probability measure $\nu$, reconstruction results
are available if the sequence of norms $(\|\phi_j\|_\infty)$ satisfies certain 
boundedness assumptions. 
We will describe this in the slightly more general 
context of weighted sparsity \cite{rawa13}. 

For a weight sequence $\omega = (\omega_{j})_{j \in \Indx}$ 
satisfying $\omega_j \geq 1$, 
we introduce the weighted $\ell_{\omega,p}$-space endowed with the norm
\begin{equation}\label{transfer:weightslp}
\triple \bsx \triple_{\omega,p} 
:= 
\left(  \sum_{j \in \Indx} |x_j|^p \omega_j^{2-p} \right)^{1/p}.
\end{equation}
The exponent $2-p$ at the weight $\omega$ may seem unfamiliar, but is most convenient in our context, see also \cite{rawa13}.
Of course, one can pass to standard definitions of weighted $\ell_p$-norms via a change of the weight.

Note that $\triple x \triple_{\omega,2} = \|x\|_2$ reduces to the 
unweighted $\ell_2$ norm, while $\triple \bsz\triple_{\omega,1} = \sum_{j \in \Indx} \omega_j |z_j|$.
Moreover, formally passing with $p \to 0$,
we obtain 
\[
\triple \bsx\triple_{\omega,0} := \sum_{j \in \supp \bsx} \omega_j^2.
\]
We shall say that $\bsx\in \R^{\Indx}$ is {\em weighted $s$-sparse} 
if $\triple \bsx \triple_{\omega,0} \leq s$. 
The weighted $s$-sparse approximation error in $\ell_{\omega,p}$ 
is defined as
\[
\sigma_s(\bsx)_{\omega,p} 
= 
\inf_{\bsz : \triple \bsz\triple_{\omega,p} \leq s} \triple\bsx-\bsz\triple_{\omega,p}.
\]
A weighted version of Stechkin's estimate was shown in \cite{rawa13}, which states that
\be\label{weighted:Stechkin}
\sigma_s(\bsx)_{\omega,q} 
\leq 
C_{p,q} s^{1/q-1/p} \triple \bsx\triple_{\omega,p}, 
\quad p < q \leq 2, \quad s \geq 2 \| \omega \|_\infty, \quad C_{p,q} = 2^{1/p-1/q}.
\ee
Choosing for instance $q = 1$, 
the $s$-term approximation error decreases quickly with increasing $s$ if
$p$ is close to $0$ and if $\triple \bsx \triple_{\omega,p}$ is small (or at least finite).
In order to reconstruct a weighted sparse vector $\bsx$ from $\bsb = \bsA \bsx$, 
we use the weighted $\ell_1$-minimization program
\begin{equation}\label{weighted:l1}
\min \triple \bsz\triple_{\omega,1} \quad \mbox{ subject to }
\bsA \bsz = \bsb.
\end{equation}
This convex optimization program can be solved efficiently with a number of 
algorithms \cite[Chapter 15]{fora13}, \cite{bova04,chpo11,pabo13}.
Weighted versions of iterative hard thresholding algorithms 
\cite{fora13,jo13,fekrra14} or CoSaMP \cite{netr08,fora13} 
may be used alternatively for the reconstruction.

Weighted $\ell_1$-minimization is guaranteed to reconstruct weighted $s$-sparse vectors under a 
variant of the (by-now classical) {\em restricted isometry property} (RIP) of the matrix $\bsA$.
The weighted restricted isometry constant $\delta_{\omega,s} = \delta_{\omega,s}(\bsA)$ 
is defined to be the smallest number such that
\[
(1-\delta_{\omega,s}) \|\bsx\|_2^2 
\leq \|\bsA \bsx\|_2^2 \leq (1+\delta_{\omega,s}) \|\bsx\|_2^2  
\quad \mbox{ for all } \bsx \mbox{ with } \|\bsx\|_{\omega,0} \leq s.
\]
Informally, we say that 
{\em 
$\bsA$ satisfies the $\omega$-RIP}
if $\delta_{\omega,s}$ is small for sufficiently large $s$.
The following result has been shown in \cite[Theorem 4.5 and Corollary 4.3]{rawa13} 
generalizing the unweighted case \cite{carota06-1,ca08,cazh14,fora13}.
\begin{thm}  \label{thm:CS}
Let $\bsA \in \R^{m \times N}$ with $\delta_{\omega,s} < 1/3$ 
for $s \geq 2 \|\omega\|_\infty$. Then for $\bsx \in \R^N$ and $\bsb = \bsA \bsx + \mathbf{\xi}$
with $\|\mathbf{\xi}\|_2 \leq \tau$, the minimizer $\bsx^\sharp$ of
\[
\min \triple \bsz\triple_{\omega,1} 
   \quad     \mbox{ subject to } \| \bsA \bsz - \bsb\|_2 \leq \tau
\]
satisfies
\begin{align}
\triple \bsx - \bsx^\sharp\triple_{\omega,1} 
& \leq c_1 \sigma_s(\bsx)_{\omega,1} + d_1 \sqrt{s} \tau, 
\\
\| \bsx - \bsx^\sharp\|_{2} 
& \leq c_2 \frac{\sigma_s(\bsx)_{\omega,1}}{\sqrt{s}} + d_2 \tau.
\end{align}
\end{thm} 
No effort has been made in optimizing the constant $1/3$. 
If $\tau = 0$ and $\bsx$ satisfies $\|\bsx\|_{\omega,0} \leq s$,
then the reconstruction is exact.

We are particularly interested in the situation when the measurements are (random) 
samples of a function having an (approximately) sparse expansion
in an orthonormal system 
$\{\phi_j\}_{j \in \Indx}$ satisfying \eqref{orthogonality}
with respect to a probability measure $\nu$ on $U$.
We further assume that the samples $\bsy_\ell$, $\ell=1,\hdots,m$, 
are chosen independently at random according to 
the measure $\eta$. 
Then the sampling matrix $\bsA$ defined in 
\eqref{def:sampling:matrix} is a structured random matrix. 
In \cite{rawa13}, the following bound for the weighted 
restricted isometry property has been proven, 
generalizing the unweighted case in \cite{cata06,ruve08,ra08,ra10}.
\begin{thm}\label{thm:weightRIP} 
Let $\omega = (\omega_j)_{j \in \Indx}$ 
be a weight sequence and $\{\phi_j\}_{j \in \Indx}$ be an orthonormal system on $(U,\eta)$ 
such that
\begin{equation}\label{weight:phi}
\| \phi_j \|_\infty = \sup_{\bsy \in U} | \phi_j(\bsy)| 
\leq \omega_j \quad \mbox{ for all } j \in \Indx.
\end{equation}
Let $\bsA \in \R^{m \times N}$, $N = \# \Indx$, 
be a random draw of the sampling matrix in \eqref{def:sampling:matrix} 
generated from independent samples $\bsy_1,\hdots,\bsy_m$ distributed
according to $\eta$. 
Given $s \geq 2 \|\omega\|_\infty^2$ and $\varepsilon,\delta \in (0,1)$, 
if
\be\label{eq:mBd}
m \geq C \delta^{-2} s \max\{ \log^3(s) \log(N), \log(\varepsilon^{-1})\},
\ee
then the weighted restricted isometry constant of 
$\frac{1}{\sqrt{m}} \bsA$
satisfies $\delta_{\omega,s} \leq \delta$ with 
probability at least $1-\varepsilon$. 
The constant $C>0$ in \eqref{eq:mBd} is universal.
\end{thm}
The important difference to the unweighted case is the weaker condition \eqref{weight:phi}, 
which allows the $L_\infty$-norms of the $\phi_j$ to grow with $j$. 
By \eqref{weight:phi} such growth requires to adapt the weights accordingly. 
As we will see below, 
the above generalization of the previous results in \cite{cata06,ruve08,ra08,ra10} 
is crucial for our application 
in the setting of tensorized \Tscheb polynomials.
A combination of the weighted RIP bound of Theorem~\eqref{thm:weightRIP} 
gives an approximation result 
in the finite-dimensional setting (finite index set $\Indx$). 
We refer to \cite[Theorem~1.1]{rawa13}.

For our purposes, we need to extend to infinite-dimensional index sets $\Indx$. 
As already mentioned in the introduction, 
the idea is to use a weighted $\ell_p$-assumption in order
to first determine a suitable finite dimensional subset 
$\Indx_0 \subset \Indx$ (where $\Indx_0$ may still be large) 
and then work on $\Indx_0$ in order to apply sparse reconstruction 
via compressive sensing. 
Here, the samples of $u$ can be interpreted
as perturbed samples of the finite-dimensional approximation of $u$ 
with elements of the orthonormal system indexed by $\Indx_0$.
Two approximation results of slightly different nature have been proven in \cite[Theorems~6.1 and 6.4]{rawa13}. 
Both work with the index set
\begin{equation}\label{inf:index:set}
\Indx_0^s = \{ j : \omega_j^2 \leq s/2\}.
\end{equation}
We will use arguments of their proofs for showing our main result, 
Theorem~\ref{thm:main:func} below, and refer to \cite{rawa13}
for details about the original results.
\section{Weighted $\ell_p$-estimates for \Tscheb expansions}
\label{sec:WeighEllp}
We now return to the parametric operator equation \eqref{eq:main}, i.e.,
\begin{equation}\label{eq:main2}
  A(\bsy)\, u (\bsy) = f \quad \mbox{ with }\quad  
  A(\bsy) = A_0 + \sum_{j\ge 1} y_j A_j, \quad \bsy \in U\;.
\end{equation}
We will use the expansion \eqref{Cheb:expand} of the solution $u(\bsy)$ 
in terms of tensorized \Tscheb polynomials.
The previous section, and in particular \eqref{weighted:Stechkin}, 
motivates to study whether the expansion
coefficients $d_{\bnu} \in \cX$ in 
\eqref{Cheb:expand} satisfy 
$(\| d_{\bnu} \|_{\bcX})_{\nu \in \cF} \in \ell_{\omega,p}(\cF)$ 
for a suitable choice of the weight sequence 
$(\omega_{\bnu})_{\bnu \in \cF}$ under certain
assumptions on the sequence $(\|\psi_j\|_\infty)_{j \in \N}$.
\subsection{Analytic continuation of parametric solutions to the complex domain}
\label{sec:AnCont}
In order to develop weighted $\ell_p$-bounds we rely on 
analytic continuation of the parametric solution map $\bsy \mapsto u(\bsy)$ 
as in \cite{codesc10-1,codesc11,Hanse486}.
To this end we consider complex parameter sequences
$\bsz = (z_j)_{j\geq 1}$ where 
$z_j = y_j + \bsi w_j$, $\bsi := \sqrt{-1}\in \C$.
For a radius $r>0$, we denote by 
$\cD_r = \{ z\in \C : |z| \leq r \}$ 
the (closed) disc of radius $r$ in $\C$ centered at the origin. 
We denote, for a sequence $\bsrho = (\rho_j)_{j\geq 1}$,
by 
$\cD_\bsrho := \prod_{j\in \N} \cD_{\rho_j}\subset \C^\N$ 
the corresponding polydisc. 
All function spaces $\bcX$, $\bcY$, etc.\ are now understood
as their ``complexifications'', i.e., as spaces over the coefficient
field $\C$. 

The complex extension of the parametric problem
\eqref{eq:parmOpEq} reads: given $\bsz \in \C^\N$,
\begin{equation} \label{eq:parmOpEqz}
 \mbox{find} \quad u(\bsz) \in \bcX \mbox{ such that } \quad
 \fa(\bsz;u(\bsz), w) \,=\,  {_{\bcY'}}\langle  f,w \rangle_{\bcY}
 \quad \mbox{ for all } w \in \bcY.
\end{equation}
Using Assumption \ref{ass:AssBj}, we may write, for every $\bsz\in \C^\N$,
$A(\bsz) = A_0(I + \sum_{j\geq 1}z_j A_0^{-1}A_j)$.
A Neumann series argument then shows that $A(\bsz)$ is boundedly
invertible for every $\bsz \in \cD_\bsrho$ under condition 
\eqref{eq:Bjsmall}, provided that 
\begin{equation}\label{delta:admissible}
\sum_{j\geq 1} \rho_j \beta_{0,j} \leq 1 - \delta\;, 
\quad \mbox{ for some } \quad 0 < \delta < 1-\kappa.
\end{equation}
A sequence $\bsrho$ will be called $\delta$-{\em admissible} 
if 
this condition holds.
Then, for $\bsz \in \cD_{\bsrho}$, the complex-parametric 
problem \eqref{eq:parmOpEqz} admits a unique solution $u(\bsz)$ which satisfies with
$\delta \mu_0 \leq \mu = (1-\kappa)\mu_0$  (cp.\ \eqref{eq:apriori})
the uniform a-priori estimate
\begin{equation}\label{eq:aprioriz}
\sup_{\bsz\in \cD_\bsrho} \| u(\bsz) \|_{\bcX} \,\leq\, \frac{\| f \|_{\bcY'}}{\delta\mu_0}. 
\end{equation}
To see this, we
fix in $\bsz$ all components except $z_k$, say.
Then, the affine parameter dependence of $A(\bsz)$ 
implies that for $\bsz \in \cD_\bsrho$ the parametric solution
$u(\bsz) = (A(\bsz))^{-1}f$ is holomorphic with respect to $z_k\in \cD_{\rho_k}$,
being the image of $f$ under a resolvent operator since, with $B_j := A_0^{-1}A_j$,
we have
\[
u(\bsz) = \left( (I + \sum_{j\ne k} z_j B_j) + z_k B_k \right)^{-1} A_0^{-1} f \;.
\]
Together with \eqref{eq:B0infsup}, this identity also implies  the bound \eqref{eq:aprioriz},
since for every $\bsz\in \cD_\bsrho$
\[
\| u(\bsz) \|_\bcX 
=
\left\| \left( I + \sum_{j\geq 1} z_j B_j\right)^{-1} A_0^{-1} f \right\|_{\bcX}
\leq 
\left\|  \left( I + \sum_{j\geq 1} z_j B_j\right)^{-1} \right\|_{\cL(\bcX,\bcX)}
\| A_0^{-1} f \|_{\bcX}
\leq \frac{\| f \|_{\bcY'}}{[1-(1-\delta)] \mu_0}.
\]
Assumption~\ref{ass:AssBj} implies that
the constant sequence $\rho_j = 1$ is $\delta$-admissible
for $0 < \delta \leq 1-\kappa$ and that for $0<\delta < 1-\kappa$, 
there exist $\delta$-admissible sequences with $\rho_j > 1$ for 
every $j \geq 1$ so that $U \subset \cD_\bsrho$ with strict inclusion, in each variable.
%
\subsection{Bounds for \Tscheb coefficients}
\label{sec:BndTscheb}
Let us now consider the tensorized \Tscheb expansion \eqref{Cheb:expand} 
of the parametric solution $u(\bsy) \in \bcX$ of the 
affine parametric operator equation \eqref{eq:main}, i.e.,
\begin{equation}\label{Cheb:expand2}
u(\bsy) = \sum_{\bnu \in \cF} d_{\bnu} T_{\bnu}(\bsy)\qquad \mbox{ for all } \bsy\in U\;.
\end{equation}
Such an expansion is valid with unconditional convergence in $L_2(U;\eta)$ where
$\eta$ denotes the countable product \Tscheb measure \eqref{prod:Cheb:meas}, 
see e.g., \cite{Hanse486}.
The following estimate on the norms of the coefficients $d_{\bnu} \in \bcX$ 
will be crucial for us, see also \cite[Proposition 5.2]{Hanse486}. 
Below, we will use the usual notation
$\rho^{\bnu} := \prod_{j \geq 1} \rho_j^{\nu_j}$ with the understanding that $0^0 = 1$.
\begin{prop}\label{prop:dnu:bound} 
Let $\bsrho = (\rho_j)_{j \geq 1}$ be a $\delta$-admissible sequence.
Then
\begin{equation}\label{dnu:estimate}
\|d_{\bnu}\|_{\bcX} 
\leq 
(\delta \mu_0)^{-1} \|f\|_{\bcY'} 2^{\|\bnu\|_0/2} \rho^{-\bnu}.
\end{equation}
\end{prop}
\begin{proof} 
We proceed similarly to \cite[Section 3]{ri90-1} and \cite{tr13}. 
By orthonormality \eqref{Tnu:orth} of the $T_{\bnu}$, 
\[
d_\bnu = \int_{U} u(\bsy) T_{\bnu}(\bsy) d\eta(\bsy) \in \bcX.
\]
For $\bnu = 0$, the bound follows from 
\[
\|d_0\|_V = \left\| \int_U u(\bsy) d \eta(\bsy) \right\|_V \leq \max_{\bsy \in U} \|u(\bsy)\| \leq (\delta \mu_0)^{-1} \|f\|_{\bcY'},
\]
where we have used the a-priori bound \eqref{eq:aprioriz} and the fact that $\eta$ is a probability measure.
Let us assume now that
$\bnu = n \mathbf{e}_1 = (n,0,0,0,\hdots)$ with $n \in \N$. 
Then writing $U = [-1,1] \times U'$ and $u(\bsy) = u(y_1,\bsy')$ 
with $\bsy'=(y_2,y_3,\hdots)$, 
we have
\[
d_{n \mathbf{e}_1} 
= 
\int_{\bsy' \in U'} \int_{-1}^1 T_n(t)  u(t,\bsy') \frac{dt}{\pi\sqrt{1-t^2}} d\eta(\bsy').
\]
By a change of variables and the definition \eqref{eq:DefTj} of the normalized Chebyshev polynomials we obtain
\begin{align}
\int_{-1}^1 T_n(t)  u(t,\bsy') \frac{dt}{\pi\sqrt{1-t^2}} 
& = \frac{\sqrt{2}}{\pi} \int_0^\pi u(\cos(\phi), \bsy') \cos(n \phi) d\phi 
= \frac{\sqrt{2}}{2\pi} \int_{-\pi}^\pi u(\cos(\phi),\bsy') \cos(n \phi) d\phi 
\notag\\
& = \frac{\sqrt{2}}{2\pi i} \int_{|\zeta|=1} 
u\left( \frac{\zeta + \zeta^{-1}}{2},\bsy'\right) 
\frac{\zeta^n + \zeta^{-n}}{2} \frac{d\zeta}{\zeta}
\notag\\
& = \frac{\sqrt{2}}{4\pi i} \int_{|\zeta|=1} u\left( \frac{\zeta + \zeta^{-1}}{2},\bsy'\right) 
\zeta^{n-1} d\zeta + 
\frac{\sqrt{2}}{4 \pi i} \int_{|\zeta|=1} 
u\left( \frac{\zeta + \zeta^{-1}}{2},\bsy'\right) \zeta^{-n-1} d\zeta,\notag
\end{align}
where we have applied the transformation $\zeta = e^{i \phi}$. 
The Joukowsky map
$\zeta \mapsto {\mathcal J}(\zeta) =  (\zeta + \zeta^{-1})/{2}$
maps the unit circle $\{\zeta \in \C : |\zeta| = 1\}$ 
onto the interval $[-1,1]$ (traversed twice) and more generally, 
both circles $S_\sigma = \{\zeta \in \C : |\zeta| = \sigma\}$ and 
$S_{\sigma^{-1}}$ for $\sigma > 1$ onto the Bernstein ellipse 
$B_\sigma = \{(\zeta + \zeta^{-1})/2 : |\zeta| = \sigma\}$.
Furthermore, for $\sigma > 1$, it maps the annulus 
$A_\sigma = \{\zeta \in \C : \sigma^{-1} \leq |\zeta| \leq \sigma\}$ 
onto the region $E_\rho$ bounded
by the Bernstein ellipse $B_\rho$. 
As outlined in the previous section, 
the function $z_1 \mapsto u(z_1,\bsy')$ is analytic on the disc 
$D_{\rho_1} = \{z_1 \in \C : |z_1| \leq \rho_1\}$ and since 
$E_{\rho_1} \subset D_{\rho_1}$ it is in particular analytic on $E_{\rho_1}$. 
Therefore, by the previous remarks, the functions
\[
\zeta \mapsto u\left( \frac{\zeta + \zeta^{-1}}{2},\bsy'\right) \zeta^{n-1} \quad \mbox{ and } \quad 
\zeta \mapsto u\left( \frac{\zeta + \zeta^{-1}}{2},\bsy'\right) \zeta^{-n-1}
\]
are analytic on the annulus $A_{\rho_1}$. 
Hence, by Cauchy's theorem, we have, for any $1 < \sigma < \rho_1$,
\begin{align*}
& \int_{-1}^1 T_n(t)  u(t,\bsy') \frac{dt}{\pi\sqrt{1-t^2}}\\
&=  \frac{\sqrt{2}}{4 \pi i} \int_{|\zeta|=\sigma^{-1}} u\left( \frac{\zeta + \zeta^{-1}}{2},\bsy'\right) \zeta^{n-1} d\zeta 
+ \frac{\sqrt{2}}{4 \pi i}  \int_{|\zeta|=\sigma} u\left( \frac{\zeta + \zeta^{-1}}{2},\bsy'\right) \zeta^{-n-1} d\zeta. 
\notag
\end{align*}
Employing the a priori bound \eqref{eq:aprioriz} 
we obtain
\begin{align}
\left\| \int_{-1}^1 T_n(t)  u(t,\bsy') \frac{dt}{\pi\sqrt{1-t^2}} \right\|_{\bcX} 
& \leq \frac{\sqrt{2} 2 \pi \sigma^{-1}}{4\pi} \frac{\|f \|_{\bcY'}}{\delta \mu_0} \sigma^{-n+1} 
+ 
\frac{\sqrt{2} 2 \pi \sigma}{4\pi} \frac{\|f \|_{\bcY'}}{\delta \mu_0}  \sigma^{-n-1}
\notag\\
& =\frac{\|f \|_{\bcY'}}{\delta \mu_0}  \sqrt{2} \sigma^{-n}.
\end{align}
Since this holds for any $\sigma < \rho_1$ and since $\eta$ 
is a probability measure we obtain by another integration
\[
\|d_{n\mathbf{e}_1}\|_V \leq (\delta \mu_0)^{-1} \|f \|_{\bcY'} \sqrt{2} \rho_1^{-n}.
\]
For general $d_{\bnu}$, with $\bnu \in \cF$, we
reason analogously: given $\bnu\in \cF$,
Cauchy's integral theorem is iterated in the (finitely many) variables 
$\{ z_j \in \C: \nu_j \ne 0 \}$.
\end{proof}

With these tools at hand we can now consider weighted $\ell_p$ estimates of the coefficients $d_{\bnu}$.
We introduce a weight sequence $\bsv = (v_j)_{j \geq \N}$ with $v_j \geq 1$ 
on the natural numbers. 
We strengthen Assumption~\ref{ass:AssBj} by requiring that there
exists a constant $0 < \kappa_{v,p} < 1$ such that
\begin{equation} \label{eq:Bjsmall:weight} 
 \sum_{j\geq 1} \beta_{0,j} 
 v_j^{(2-p)/p} \leq \kappa_{v,p} \;,
 \quad\mbox{where}\quad
 \beta_{0,j} \,:=\, \| A_0^{-1} A_j \|_{\cL(\cX,\cX)}\;,
 \quad j=1,2,\ldots,
\end{equation}
and that
\begin{equation}\label{v:lp}
 \sum_{j\geq 1} \beta_{0,j}^p v_j^{2-p} 
 < \infty.
\end{equation}
Since $v_j \geq 1$, these assumptions imply \eqref{eq:Bjsmall} and \eqref{eq:assump_t} for $t=0$.
Associated to the weight $\bsv$ and a number $\theta \geq 1$, 
we introduce a weight sequence 
$\omega = (\omega_\bnu)_{\bnu \in \cF}$ on $\cF$ 
via
\begin{equation}\label{weightF}
\omega_{\bnu} = \prod_{j\in \supp \bnu} \theta v_j^{\nu_j} = \theta^{\| \bnu \|_0} \bsv^{\bnu}.
\end{equation}
We have the following bound on the weighted $\ell_p$-summability of the coefficient sequence $(d_{\nu})_{\nu \in \cF}$, 
extending main results from \cite{codesc10-1,codesc11,Hanse486}.

\begin{thm}\label{thm:weightlp} Let $0 < p \leq 1$. Assume that \eqref{eq:Bjsmall:weight} and \eqref{v:lp} hold for some
weight sequence  $\bsv \geq 1$.
For $\theta \geq 1$ construct a weight sequence $\omega$ on $\cF$ via \eqref{weightF}.
Then the sequence of norms of the coefficients 
$(\|d_\nu\|_{\bcX})_{\nu \in \cF}$ of the parametric
solution $u(\bsy)$ in the 
tensorized \Tscheb expansion \eqref{Cheb:expand2} is contained
in $\ell_{\omega,p}(\cF)$, 
i.e., $\sum_{\bnu \in \cF} \omega_{\bnu}^{2-p} \|d_{\bnu}\|_{\bcX}^p < \infty$.
\end{thm}
The proof is based on the following observation in \cite{codesc10-1}, 
where we use the convention that for $\bnu \in \cF$ 
we define $\bnu! = \prod_{j \in \supp \bnu} \nu_j!$ and $| \bnu | = \sum_{j \geq 1} \nu_j$. 
Note that $| \bnu |! \geq \bnu !$.
\begin{thm}\label{thm:lp} 
For $0<p \leq 1$ and a sequence $(a_j)_{j \geq 1}$, we have 
$\left(\frac{|\bnu|!}{\bnu!} a^{\bnu}\right)_{\bnu \in \cF} \in \ell_p(\cF)$ 
if and only if $\sum_{j\geq 1} a_j < 1$ and $(a_j)_{j \geq 1} \in \ell_p(\N)$. 
\end{thm}

\begin{proof}[Proof of Theorem~\ref{thm:weightlp}] 
We proceed similarly as in \cite{codesc11}. 
The idea is to construct, for each $\bnu\in \cF$,
a suitable $\delta$-admissible sequence $\bsrho = (\rho_j)$ 
(depending, in general, on $\bnu$) with $\rho_j \geq 1$, 
where we choose $\delta = (1-\kappa_{v,p})/2$ so that by \eqref{dnu:estimate} 
\begin{equation}\label{dnu:bound2}
\|d_{\bnu}\|_{\bcX} \leq \frac{2}{\delta\mu_0} \|f\|_{\bcY'} 2^{\|\bnu\|_0/2} \prod_{j \geq 1} \rho_j^{-\nu_j}.
\end{equation}
For convenience we introduce $\tilde{v}_j = v_j^{(2-p)/p}$ and  $\tilde{\theta} = \theta^{(2-p)/p}$.
For given $\bnu \in \cF$, we construct the sequence $\bsrho$ 
by first choosing a finite index set $E \subset \N$ such that, 
for $F = \N \setminus E$,
\begin{equation}\label{F:def}
\sum_{j \in F} \tilde{v}_j \beta_{0,j} \leq \frac{\delta}{8 \tilde{\theta}}.
\end{equation}
Such a set $E$ exists by Assumption \eqref{eq:Bjsmall:weight}.
We further choose $\alpha > 1$ such that
\[
(\alpha -1) \sum_{j \in E} \tilde{v}_j \beta_{0,j} < \frac{\delta}{2}.
\]
Then we define the sequence $\bsrho$ via
\[
\rho_j = \left\{ 
\begin{array}{cl} \alpha \tilde{v}_j & \mbox{ if } j \in E,
\\
\max\left\{\tilde{v}_j, \frac{\delta \nu_j}{2 |\bnu_F| \beta_{0,j}} \right\} & \mbox{ if } j \in F
\;.
\end{array} \right.,
\]
where $|\bnu_F| = \sum_{j \in F} \nu_j$.
The sequence $\bsrho$ is $\delta$-admissible since
\begin{align}
\sum_{j \geq 1} \rho_j \beta_{0,j}
& = \sum_{j \in E} \alpha \tilde{v}_j \beta_{0,j}
+ 
\sum_{j \in F} \max\left\{\tilde{v}_j, \frac{\delta \nu_j}{2 |\bnu_F| \beta_{0,j}} \right\} \beta_{0,j} 
\notag 
\\
& \leq (\alpha - 1) \sum_{j \in E} \tilde{v}_j \beta_{0,j} 
+
\sum_{j \in E} \tilde{v}_j \beta_{0,j} + \sum_{j \in F} \tilde{v}_j \beta_{0,j}  + \frac{\delta}{2} \notag
\\
& < \frac{\delta}{2} + \kappa_v + \frac{\delta}{2} = \kappa_v + \delta = 1-\delta.  \notag
\end{align}
Here, we have used that $v_j \geq 1$ together with \eqref{eq:Bjsmall:weight}, 
and furthermore that $\kappa_v = 1- 2\delta$ by the choice of $\delta$.
Therefore, the bound \eqref{dnu:bound2} is valid and implies
\[
\|d_{\bnu}\|_V \leq C_\delta 2^{\|\bnu\|_0/2}\prod_{j \in E} (\alpha \tilde{v}_j)^{-\nu_j} 
\prod_{j \in F} \min \left\{ \tilde{v}_j^{-\nu_j}, \left(\frac{|\bnu_F| g_j}{\nu_j}  \right)^{\nu_j}\right\}
\]
with $C_\delta= 2 (\delta \mu_0)^{-1} \|f\|_{V^*}$ and
\[
g_j = 2 \delta^{-1} \beta_{0,j}.
\]
Above we adopt the convention that a factor equals $1$ if $\nu_j = 0$.

Now we estimate the weighted $\ell_p$-norm with the weight $\omega_{\bnu}$. 
For convenience we introduce 
$\tilde{\omega}_{\bnu} 
 = \omega_{\bnu}^{(2-p)/2} = \tilde{\theta}^{\|\bnu\|_0} \prod_{j \geq 1} \tilde{v}_j^{\nu_j}$.
We let $\cF_F$ be the finitely supported sequences of natural numbers (including $0$) 
indexed by $F$ and likewise define $\cF_E$. We obtain
\begin{align}
\sum_{\bnu \in \cF} \omega_{\nu}^{2-p} \| d_{\bnu} \|_{\bcX}^p  & = \sum_{\bnu \in \cF} \tilde{\omega}_{\nu}^p \| d_{\bnu} \|_{\bcX}^p
 \leq C_\delta^p \sum_{\bnu} 2^{\|\bnu\|_0/2} \tilde{\theta}^{\|\bnu\|_0} \left( \prod_{j \in E} \tilde{v}_j^{p\nu_j} (\alpha \tilde{v}_j)^{-p\nu_j} \right) \left( \prod_{j \in F} \tilde{v}_j^{p \nu} \min \left\{ \tilde{v}_j^{-p\nu_j}, \left(\frac{|\bnu_F| g_j}{\nu_j}  \right)^{p\nu_j}\right\} \right) \notag\\
& \leq C_\delta^p \left( \sum_{\bnu \in \cF_E} (\sqrt{2}\tilde{\theta})^{\|\bnu\|_0} \prod_{j \in E} \alpha^{-p\nu_j}\right) \left(\sum_{\bnu \in \cF_F} (\sqrt{2} \tilde{\theta})^{\|\bnu\|_0} \prod_{j \in F} \left( \frac{|\bnu_F| \tilde{v}_j g_j}{\nu_j}  \right)^{p\nu_j} \right). \label{psum_factors}
\end{align}
Let us continue with the first factor above,
\[
 \sum_{\bnu \in \cF_E} (\sqrt{2} \tilde{\theta})^{\|\bnu\|_0} \prod_{j \in E} \alpha^{-p\nu_j} = \prod_{j \in E} \left(1 + \sqrt{2} \tilde{\theta} \sum_{n=1}^\infty \alpha^{-pn} \right)
 = \left(1+ \frac{\sqrt{2}\tilde{\theta} \alpha^{-p}}{1-\alpha^{-p}}\right)^{\#E}.
\]
For the second factor in \eqref{psum_factors}, it follows as in \cite[Section~3.2]{codesc11} 
(via Stirling's formula) that
\[
(\sqrt{2}\tilde{\theta})^{\|\bnu\|_0} \prod_{j \in F} \left( \frac{|\bnu_F| v_j g_j}{\nu_j}  \right)^{p\nu_j} 
\leq \frac{|\bnu_F|!}{\bnu!} \prod_{j \in F} (\tilde{v}_j g_j)^{\nu_j} \max\{1,\sqrt{2}\tilde{\theta} e \sqrt{\nu_j}\}
\leq \frac{|\bnu_F|!}{\bnu!} h^{\nu_F},
\]
where 
\[
h_j = e \sqrt{2} \tilde{\theta} v_j g_j = 2 \sqrt{2} e \delta^{-1} \tilde{\theta} \tilde{v}_j \beta_{0,j}. 
\]
By \eqref{F:def} we have
\[
\sum_{j \in F} h_j \leq \frac{2\sqrt{2}e}{8} < 1.
\]
Since $\mathbf{h} = (h_j)_{j \geq 1} \in \ell_p(\N)$ by Assumption~\eqref{v:lp}, 
it follows from Theorem~\ref{thm:lp} that the sequence
$(\frac{|\bnu_F|!}{\bnu!} \mathbf{h}^{\bnu_F})_{\bnu \in F}\in \ell_p(\cF_F)$ 
so that also the second factor in \eqref{psum_factors} is finite and, 
hence, $(\| d_{\bnu} \|_{\bcX} \tilde{\omega}_{\bnu})_{\bnu \in \cF} \in \ell_p(\cF)$ which means that 
$\sum_{\bnu \in \cF} \| d_{\bnu} \|_{\bcX}^p \omega_\bnu^{2-p} < \infty$. 
\end{proof}
\begin{remark}\label{rem:diff:eq} 
In case of the parametric diffusion equation \eqref{diff:eq:weak},
assumption \eqref{eq:Bjsmall:weight} 
may be relaxed to a weighted version of the uniform ellipticity assumption, i.e.,
\[
\sum_{j \geq 1} v_j^{(2-p)/p} |\psi_j(x)|  \leq \min\left\{\bar{a}(x) - r, R - \bar{a}(x) \right\} \quad \mbox{ for all } x \in D,
\]
while \eqref{v:lp} is replaced by
\[
\sum_{j \geq 1} v_j^{2-p} \|\psi_j\|_\infty^p < \infty \quad \mbox{ for some } 0 < p \leq 1.
\]
Under these conditions, the coefficients 
$d_{\bnu} \in \cX = H_0^1(D)$ in the \Tscheb expansion \eqref{Cheb:expand:sol} 
of the solution of the parametric equation satisfy 
$\sum_{\bnu} \|d_{\bnu}\|_V \omega_{\bnu}^{2-p}< \infty$ 
with the weights $\omega_{\bnu}$ given by \eqref{weightF}. 
This fact is shown analogously to the above proof, see also
\cite{codesc11} for unweighted $\ell_p$-summability 
in the diffusion equation context.
\end{remark}
\begin{remark} 
\label{remk:ChoiWgt}
We comment on possible choices for the weights.
\begin{itemize}
\item 
With the trivial weight $v_j = 1$, the above result generalizes the one from \cite{codesc11} 
in the sense
that $(\beta_{0,j}) \in \ell_p(\N)$ (or 
$(\| \psi_j\|_\infty) \in \ell_p(\N)$ in case of the diffusion equation) 
implies that $(\|d_{\bnu}\|_V)_{\bnu \in \cF}$ 
is contained in the weighted space $\ell_{\omega,p}(\cF)$
with $\omega_{\nu} = \theta^{\| \bnu\|_0}$, for any $\theta \geq 1$. 
As this weight grows exponentially with the 
number of nontrivial components in $\bnu$, the coefficients of $T_{\bnu}$ 
with many factors in the tensor product are unlikely to contribute much to the expansion.
\item 
The unweighted condition~\eqref{eq:Bjsmall} for some $\kappa<1$ already 
implies a weighted version. 
In fact, we may ``squeeze in'' weights of the form
$v_j = (1+ \tau)$ with $\tau>0$ sufficiently small 
so that the weighted summability condition \eqref{eq:Bjsmall:weight} 
holds with some $\kappa_{v,p}$ satisfying $\kappa < \kappa_v < 1$.
For this weight, if $(\beta_{0,j}) \in \ell_p(\N)$, 
then $(\beta_{0,j}) \in \ell_{p,v}(\N)$ 
so that Theorem~\ref{thm:weightlp} states that 
$(\|d_{\bnu}\|_V)_{\bnu \in \cF} \in \ell_{\omega,p}(\cF)$ 
for the weight
\[
\omega_{\bnu} = \theta^{\|\bnu\|_0} (1+\tau)^{\|\bnu\|_1}.
\]
\item Polynomial decay of the sequence $(\beta_{0,j})$ allows
to ``squeeze in'' a polynomially growing sequence of weights $v_j \geq 1$.
Suppose that $|\beta_{0,j}| \leq c j^{-t}$ for some $t > 1$ 
and sufficiently small $c> 0$ so that
\[
\sum_{j \geq 1} \beta_{0,j} \leq c \sum_{j \geq 1} j^{-t} \leq c \zeta(t) < 1,
\]
which means that the unweighted condition \eqref{eq:Bjsmall} is satisfied.
Now let $v_j = \gamma j^{\tau}$ for some $0 < \tau < t-1$ and $\gamma > 1$. 
Then
\[
\sum_{j \geq 1} v_j \beta_{0,j} \leq c \gamma \sum_{j \geq 1} j^{-(t-\tau)} = c \gamma \zeta(t-\tau).
\]
If $\tau$ and $\gamma$ are such that $\zeta(t-\tau) < (c\gamma)^{-1}$ (which is possible by $\zeta(t) < c^{-1}$ and continuity of the $\zeta$-function)
then the weight sequence $(v_j)$ is a valid choice as $\sum_{j \geq 1} \beta_{0,j} v_j \leq \kappa_{v} < 1$.
The resulting weight $\omega$ on $\cF$, i.e.,
\[
\omega_{\bnu} 
= \theta^{\|\bnu\|_0} \prod_{j \in \supp \bnu} v_j^{\bnu_j} 
= \theta^{\|\bnu\|_0} \prod_{j \in \supp \bnu}\gamma^{\bnu} j^{\tau \bnu_j}
\] 
growths polynomially with respect to the dimension $j$ 
and exponentially with respect to each 
$\nu_j$.
\item Similarly, exponential decay of the sequence $(\beta_{0,j})$
allows to choose an exponentially 
growing sequence of weights $v_j \geq 1$.
Assume that 
$\beta_{0,j} \leq c \alpha^j$ for some $0 < \alpha  < 1$ and that
$\sum_{j \geq 1} \beta_{0,j} \leq c \sum_{j \geq 1} \alpha^j = c \frac{\alpha}{1-\alpha} = \kappa < 1$
implying the unweighted condition \eqref{eq:Bjsmall}. 
Choosing $v_j := \sigma^j$ for some $\sigma > 1$ with $\sigma \alpha < 1$, 
we have
\[
\sum_{j \geq 1} v_j \beta_{0,j}
\leq c \sum_{j \geq 1} (\sigma \alpha)^j 
= c \frac{\sigma \alpha}{1-\sigma \alpha}.
\]
If $\sigma$ is sufficiently close to $1$, then 
$\sum_{j \geq 1} v_j \beta_{0,j} \leq \kappa_v < 1$ 
so that the weighted condition \eqref{eq:Bjsmall:weight} is satisfied.
The corresponding weight $\omega$ is defined as
\[
\omega_{\bnu} = \theta^{\|\bnu\|_0} \prod_{j \in \supp \bnu} \sigma^{j \nu_j},
\qquad 
\bnu \in \cF.
\]
\end{itemize}
\end{remark}

\begin{remark}\label{rem:other:pol} 
At this stage it is appropriate
to discuss why we prefer to work with tensorized \Tscheb polynomials rather
than other polynomial systems. Clearly, Taylor monomials are immediately 
ruled out because they do not form an orthonormal system
(with respect to any measure), so that the setup described above does not apply.
One-dimensional Legendre polynomials $L_j$ have the disadvantage that their $L_\infty$-norms 
grow as $\|L_j\|_\infty = \sqrt{2j+1}$, 
see e.g.\ \cite{rawa12}, so that the tensorized Legendre polynomials 
$L_{\bnu}(\bsy) = \prod_{j \in \supp_{\bnu}} L_{\nu_j}(y_j)$ yield
\[
\|L_{\bnu}\|_\infty = \prod_{j \in \supp {\bnu}} \sqrt{2\nu_j+1}.
\]
In principle, we can compensate for that by using weights $v_j$, $j \geq 1$, such that 
\begin{equation}\label{cond:Lbnu}
\omega_\bnu = \theta^{\| \bnu \|_0/2} \prod_{j \in \supp \bnu} v_j^{\nu_j} \geq \|L_{\bnu}\|_\infty, 
\end{equation}
see Theorem~\ref{thm:weightRIP} (it seems that an analog of Theorem~\ref{thm:weightlp} 
introducing these weights $\omega_{\bnu}$ 
also holds for tensorized Legendre polynomials, but details are not worked out yet).
However, \eqref{cond:Lbnu} 
puts stronger conditions on the set of admissible weight sequences 
$(v_j)_{j \geq 1}$ and thereby on the operators $A_j$ than required for \Tscheb polynomials. 

In the one-dimensional (or low dimensional) case, one may alternatively apply the preconditioning trick from \cite{rawa12} 
to overcome the problem of growing $L_\infty$-norms of the
Legendre polynomials. This demands to work with the premultiplied functions $Q_j(y) = \sqrt{\pi/2}(1-y^2)^{1/4}L_j(y)$ which satisfy
the nice uniform bound $\|Q_j\|_\infty \leq \sqrt{3}$ for all $j \in \N_0$, see \cite{rawa12}. 
However, in the $d$-dimensional case the functions $Q_{\bnu}(\bsy) = \prod_{j=1}^d Q_{\nu_j}(y_j)$ satisfy
$\|Q_j\|_\infty \sim \gamma^d$ for some $\gamma > 1$, essentially because $\|Q_0\|_\infty = \sqrt{\pi/2}$, and in particular in the case $d= \infty$,
the infinite product defining $Q_{\bnu}$ for $\bnu \in \cF$ does not even converge, which prohibits the preconditioning trick in the high or 
infinite-dimensional case.
\end{remark}

\section{Compressive Sensing Petrov-Galerkin Discretization}
\label{sec:CS}
Let us now return to our main goal of approximating the solution 
of the affine parametric equation \eqref{eq:main}, 
$ A(\bsy)\, u (\bsy) = f$, with $A(\bsy) = A_0 + \sum_{j \geq 1} y_j A_j$,
via compressive sensing techniques.
Consider the Chebyshev expansion \eqref{Cheb:expand}. 
In view of Condition \eqref{weight:phi}, 
we need to choose weights $\omega$ satisfying 
$\omega_j \geq \| T_{\bnu} \|_\infty = 2^{\|\bnu\|_0/2}$.
Given a weight sequence $\mathbf{v} = (v_j)_{j \in \N}$ with $v_j \geq 1$, 
the previous section suggests to use weights of the form
\begin{equation}\label{omega:def2}
\omega_{\bnu} = \theta^{\|\bnu\|_0} \prod_{j \in \supp \bnu} v_j^{\nu_j} \quad \mbox{ with } \theta = \sqrt{2}.
\end{equation}
For the sake of simplicity, 
we consider a functional evaluation of the solution $u(\bsy)$ in this paper, i.e., 
given a bounded linear functional $G: \bcX \to \R$, 
we are interested in numerically computing an approximation of
\[
F(\bsy) = G(u(\bsy)), \quad F : U \to \R\;.
\] 
Then, the tensorized Chebyshev approximation to $F(\bsy)$ 
is given by the (unconditionally convergent) expansion
\[
F(\bsy) = \sum_{\bnu \in \cF} g_\bnu T_{\bnu}(\bsy),
\]
where $g_\bnu = G(d_\bnu)$ with $d_\bnu \in \bcX$ as in \eqref{Cheb:expand2}.
By boundedness of $G$, the sequence $\mathbf{g}$ of coefficients 
$g_{\bnu} = G(d_{\bnu}) \in \R$ satisfies $\mathbf{g} \in \ell_{\omega,p}(\cF)$
if $(\|d_\bnu\|_{\bcX})_{\bnu \in \cF} \in \ell_{\omega,p}(\cF)$.
Sufficient conditions for this inclusion were obtained
in Theorem~\ref{thm:weightlp}.
For the choice \eqref{omega:def2} and for $s>0$,
the index set \eqref{inf:index:set}, i.e.,  
$\Indx_0^s = \{\bnu \in \cF : \omega_{\bnu}^2 \leq s/2\}$, 
can be written as
\begin{equation}\label{def:Indxs}
\Indx_0^s = \{ \bnu \in \cF : \prod_{j \in \supp \bnu} \theta^2 v_j^{2\nu_j} \leq s/2 \}.
\end{equation}
This set is always finite if the 
weight sequence satisfies $v_j>1$ and monotonically grows to infinity as $j \to \infty$.

Let us now formulate the compressive sensing Petrov-Galerkin (CSPG)
algorithm for numerically computing an approximation to a functional of the solution of 
an affine-parametric operator equation of the form \eqref{eq:main}.

\medskip

\noindent
\fbox{
\begin{minipage}[ht]{0.96\textwidth}
\medskip
\begin{center}
\textbf{
Algorithm for the Approximation of a functional $G(u(\bsy))$ 
of the solution of a parametric equation via compressive sensing} \vspace{-2mm}\\
\rule{0.96\textwidth}{.8pt}
\end{center} 

{\bf Input:}
\begin{itemize}\itemsep0pt
\item Weights $(v_j)_{j \geq1}$ with $v_j \geq 1$ satisfying \eqref{eq:Bjsmall:weight} and \eqref{v:lp} for some $0 < p < 1$.
\item Accuracy $\varepsilon$ and sparsity parameter $s$ 
\item Index set $\Indx_0^s = \{ \bnu \in \cF : 2^{\|\bnu\|_0} \prod_{j \in \supp \bnu} v_j^{2\nu_j} \leq s/2 \}$
such that $N := \# \Indx_0^s < \infty$ 
\item Number of samples $m \asymp s \ln^3(s) \ln(N)$.
\end{itemize}
{\bf Algorithm}
\begin{enumerate}[1:]
\item Choose samples $\bsy_1,\hdots,\bsy_m \in U$ independently at random according to the Chebyshev product measure
\eqref{prod:Cheb:meas}. 
\item 
For given $\eps>0$, choose $h=h(\eps)>0$ and truncation level $B=B(\eps)\in \N$ 
such that the dimension-truncated Galerkin approximations 
$u_\eps(\bsy_\ell) := u^{h(\eps)}_{B(\eps)}(\bsy_\ell)
\in \bcX_h$ 
defined in \eqref{eq:uhs} to the solution samples $u(\bsy_\ell) \in \bcX$ 
of the parametric operator equation $A(\bsy_\ell) u(\bsy_\ell) = f$ 
admit the following bound of the approximation error at $\bsy_\ell$ for
$b_\ell := G(u_\eps(\bsy_\ell))$:
\begin{equation}\label{approx:sample}
|b_\ell - G(u(\bsy_\ell))| 
= | G(u(\bsy_\ell) - u_\eps(\bsy_\ell)) | 
\leq \eps.
\end{equation}
\item 
With the sampling matrix
\begin{equation}\label{sampling:matrix:Cheb}
\A_{\ell, \bnu} = T_{\bnu}(\bsy_\ell), \quad \ell =1, \hdots, m, \quad \bnu \in \Indx_0^s,
\end{equation}
and the weights 
$\omega_{\bnu} = 2^{\| \bnu \|_0/2} \prod_{j \in \supp \bnu} v_j^{\nu_j}$,
compute the solution $\mathbf{g}^\sharp \in \R^{\Indx_0^s}$ of the weighted 
$\ell_1$-minimization program
\begin{equation}\label{l1:minh}
\min \| \mathbf{g} \|_{\omega,1} \quad \mbox{ subject to } \| \bsA \mathbf{g} - {\mathbf{b}} \|_2 \leq 2 \sqrt{m} \varepsilon. 
\end{equation}
Here, $\| \mathbf{g} \|_{\omega,1} = \sum_{\bnu \in \Indx_0^{s}} \omega_\bnu |g_{\bnu}| $.
\item 
Output approximation $\widehat{F}_\eps(\bsy)$ to $F(\bsy) = G(u(\bsy))$:
\begin{equation}\label{F:sol}
\widehat{F}_\eps(\bsy) = \sum_{\bnu \in \Indx_0^s} g_{\bnu}^\sharp T_{\bnu}(\bsy).
\end{equation}
\end{enumerate}
\end{minipage}
}

\medskip

If the sparsity and accuracy parameters are set accordingly, 
we obtain the following error estimates. 
\begin{thm}\label{thm:main:func} 
Consider $u(\bsy)$, the parametric solution to the affine 
parametric equation 
$A(\bsy) u(\bsy) = f$ with $A(\bsy) = A_0 + \sum_{j\geq 1} y_j A_j$,
satisfying Assumption \ref{ass:AssBj} and \ref{ass:XtYt} with some $p=p_0 \in (0,1)$.
Let $v = (v_j)_{j \geq 1}$ be a weight satisfying  
\eqref{eq:Bjsmall:weight} and \eqref{v:lp} and let $G : \bcX \to \R$ be a bounded linear functional.
Then the expansion coefficients $\mathbf{g}=(g_{\bnu})_{\bnu \in \cF}$ of 
$F(\bsy) = G(u(\bsy)) = \sum_{\bnu \in \cF} g_{\bnu} T_{\bnu}(\bsy)$
satisfy $\mathbf{g} \in \ell_{\omega,p}(\cF)$ with weight
\[
\omega_{\bnu} = 2^{\| \bnu \|_0 / 2} \prod_{j \in \supp \bnu} v_j^{\nu_j}, \quad \bnu \in \cF.
\]
Let $\eps > 0$ be an accuracy parameter and 
assume that the sparsity parameter $s$ satisfies the condition
\be\label{eq:Condseps}
\sqrt{5}\cdot 4^{1-1/p} s^{1/2-1/p} \triple g_{\bnu} \triple_{\omega,p} 
\leq \varepsilon 
\leq C_2 s^{1/2-1/p} \triple g_{\bnu} \triple_{\omega,p}
\;,
\ee
where $C_2 > \sqrt{5}\cdot 4^{1-1/p}$ is a constant that is independent of $s$.
Let further 
$\Indx_0^s := \{ \bnu \in \cF : 2^{\|\bnu\|_0} \prod_{j \in \supp \bnu} v_j^{2\nu_j} \leq s/2 \}$ 
be such that $N := \# \Indx_0^s < \infty$. 
Draw $m$ sampling points independently, identically distributed 
according to the product measure $\eta$, 
where
\begin{equation}\label{m:samples}
m \asymp C s \log^3(s) \log(N),
\end{equation}
for a universal constant $C>0$.
Let $\widehat{F}_{\eps} : U \to \R$ 
denote the function computed via \eqref{l1:minh} and \eqref{F:sol}. 
Then 
there exists a universal constant $C'>0$ such that
with probability at least $1-2N^{-\log^3(s)}$, 
the computed function 
$\widehat{F}_\eps 
: U \to \R$ satisfies
\begin{align}
\label{err:bound:L2}
\| \widehat{F}_{\eps}- F\|_2 
& \leq C' \triple \mathbf{g} \triple_{\omega,p} s^{1/2-1/p} 
\leq C''_{\mathbf{g}} \left(\frac{\log^3(m)\log(N)}{m}\right)^{1/p-1/2}, 
\\
\label{err:bound:Linf}
\| \widehat{F}_{\eps} - F\|_\infty 
& \leq C' \triple \mathbf{g} \triple_{\omega,p} s^{1-1/p} 
\leq C''_{\mathbf{g}} \left(\frac{\log^3(m)\log(N)}{m}\right)^{1/p-1}.
\end{align}
The constant $C'$ only depends on $C_2$, while $C''_{\mathbf{g}}$ depends on $C'$ and $\|\mathbf{g}\|_{\omega,p}$.
\end{thm}

\begin{remark} 
\begin{itemize}
\item[(a)] 
The dimension-truncated Petrov-Galerkin approximation is computed 
using the $B$-term
truncated expansion
$A(\bsy_\ell)_B := A_0 + \sum_{j =1}^B y_j A_j$. 
Moreover, the entries in the matrix $\bsA$ only require to evaluate 
$T_{\bnu}(y_\ell) = \prod_{j \in \supp \bnu} T_j((\bsy_\ell)_j)$ for $\bnu \in \Indx_0^s$, 
so that only components $(\bsy_\ell)_j$ with $j \in \{1,\hdots,B'\}$ for
$B'= \max \{ j: \exists \mu \in \Indx_0^s \mbox { with } \mu_j \neq 0\}$ are needed.
This means that in practice it is enough to sample independently from the
finite product measure
\[  
\eta^{\bar{B}} := \bigotimes_{j=1}^{\bar{B}} \frac{d y_j}{\pi \sqrt{1-y_j^2}}
\]
with $\bar{B} := \max\{B,B'\}$.
\item[(b)] 
The required error estimate $|G(u^h_J) - G(u(\bsy_\ell))| \leq \eps$ 
can be guaranteed via Proposition~\ref{prop:FEconvrate} and Theorem~\ref{thm:trunc} 
by choosing the parameters $h(\eps) > 0$ and $B(\eps)\in \N$
in the Petrov-Galerkin approximation in terms of the parameters
$t,t',p_0$ in Assumptions~\ref{ass:AssBj}, \ref{ass:XtYt} as
$$
h(\eps) \simeq \eps^{1/(t+t')}\;, \qquad B(\eps) \geq \eps^{-p_0/(1-p_0)}.
$$
Note that this demands for higher regularity assumptions 
than \eqref{eq:Bjsmall:weight}. 
If $A(\bsy)$ operates nicely also on $\bcX_t$ for some $t > 0$ 
in the sense that \eqref{eq:Regul} holds, 
then $|G(u^h(\bsy)) - G(u(\bsy))| \leq C'_t h^t$. 
For instance, \eqref{eq:Regul} is implied by \eqref{eq:betatsmall} 
which is similar to \eqref{eq:Bjsmall:weight} but works with the operator norms
on $\bcX_t$ rather than the one on $\bcX$. 
\item[(c)] 
Instead of the Petrov-Galerkin method, any other 
stable and consistent discretization scheme
for numerically approximating the functional of 
the parametric solution of $A(\bsy) u = f$ 
for fixed parameter $\bsy$
can be applied in step 2 as long as the accuracy is good enough.
We have proposed the Petrov-Galerkin method here to cover all standard, as well as
certain conforming mixed finite element discretizations, for elliptic as well as space-time
discretizations of parabolic problems.
\item[(d)] 
The computations in step 2 are easily parallelized:
the sampling points $\bsy_\ell$, $\ell,\hdots,m$, 
are drawn in a preprocessing step,
and the approximate evaluations 
$G(u(\bsy_\ell))$, $\ell=1,\hdots,m$, 
are mutually independent.
Only step 3 requires the combination of the samples. 
\item[(e)] 
For the solution of the weighted $\ell_1$-minimization problem in step 3, 
various algorithms may be used, see for instance \cite[Chapter 15]{fora13} 
or \cite{pabo13} for overviews.
Preferably, one uses a method that operates only by multiplications
with $A$ and its transpose $A^*$ and does not require more complicated operations 
such as solving linear systems. 
In this case, one can exploit fast (approximate) matrix vector multiplication 
routines that are available for Chebyshev-structured matrices \cite{dadeja97,dokupo10,postta01,po03-4}. 
Extending these routines to high-dimensional problems with favourable scaling
is, however, not straightforward.
\item[(f)] 
In order to implement the algorithm as stated, an a-priori 
estimate of $\triple h_{\bnu} \triple_{\omega,p}$ is required. 
Such an estimate may in principle be computed. 
In fact, tracing the proof of Theorem~\ref{thm:weightlp} and of the auxiliary results 
required therein, may provide an estimate in terms of the weighted 
$\ell_p$-norm of $(\| \psi_j \|_\infty)_{j \geq 1}$. 
However, this bound may be very crude. 
In practice, one may as well work rather with the equality constraint 
$\bsA \mathbf{g} = \mathbf{b}$ in \eqref{l1:minh} 
with the sampling matrix $\bsA$ defined in \eqref{sampling:matrix:Cheb}.
Although there are no rigorous bounds
available for this strategy, there are theoretical 
results \cite[Chapter 11]{fora13}, \cite{wo10-2} 
suggesting that similar estimates should be possible. 
\item[(g)] 
Alternatively to weighted $\ell_1$-minimization, 
one may use a variant of iterative hard thresholding as recovery method \cite{jo13,fekrra14}.
This may have the advantage that no a-priori estimate of the accuracy parameter 
$\varepsilon$ as in \eqref{eq:Condseps} is required.
Only the weighted sparsity parameter $s$ needs to be chosen. 
If the sample solutions $G(u(\bsy_\ell))$
are computed with some accuracy $\varepsilon$, 
then the final error estimates take the form
\begin{align*}
\| F - \hat{F}_\eps \|_{L_2(U,\eta)} & \leq C_1 s^{1/2-1/p} + C_2 \eps,
\\
\| F - \hat{F}_\eps \|_{L_\infty(U)} & \leq C_1 s^{1-1/p} + C_2 \sqrt{s} \eps.
\end{align*}
If the number $m$ of samples is given then $s$ should be chosen such that $m \sim s \log^3(s) \log(N)$. 
Alternatively, if a prescribed accuracy level is given, Stechkin's estimate \eqref{weighted:Stechkin} 
provides a guideline for chosing $s$, which in turn determines the number $m$ of sample evaluations. 
Details will be presented elsewhere.
\end{itemize}
\end{remark}

\begin{proof}[Proof of Theorem~\ref{thm:main:func}] 
The fact that $\mathbf{g} \in \ell_{\omega,p}(\cF)$ follows from Theorem~\ref{thm:weightlp}.
With $\cF_1 := \cF \setminus \Indx_0^s$, we write
\[
F(\bsy) = G(u(\bsy)) = \sum_{\bnu \in \cF} g_{\bnu} T_{\bnu}(\bsy) = F_0(\bsy) + F_1(\bsy) 
\]
with
\[
F_0(\bsy) = \sum_{\bnu \in \Indx_0^s} g_{\bnu} T_{\bnu}(\bsy), 
\quad 
F_1(\bsy) = \sum_{\bnu \in \cF_1} g_{\bnu} T_{\bnu}(\bsy),
\]
We interpret the computed samples $b_\ell$ as perturbed samples of $F_0$. 
Then, the corresponding sample error can be bounded by
\[
|b_\ell - F_0(\bsy_\ell)| \leq |b_\ell - F(\bsy_\ell)| + |F_1(\bsy_\ell)| 
\leq \varepsilon + |F_1(\bsy_\ell)|, 
\;.
\]
where we used \eqref{approx:sample}. 
It is shown via Bernstein's inequality in the proof of \cite[Theorem~6.1]{rawa13} that 
\[
\left| \frac{1}{m} \sum_{\ell=1}^m |F_1(\bsy_\ell)|^2 - \sum_{ \bnu \in \cF_1} |g_\bnu|^2\right| \leq \frac{3}{s} \sum_{\bnu \in \cF_1} |g_{\bnu}| \omega_{\bnu}
\] 
with probability at least $1-e^{-3m/(2s)}$. 
Here, also the definition of $\Indx_0^s$ is used. 
Furthermore, $\omega_{\bnu}^2 \geq s/2$ for $\bnu \notin \Indx_0^s$ implies that
\[
\left(\sum_{\bnu \in \cF_1} |g_\bnu|^2\right)^{1/2} \leq \sqrt{\frac{2}{s}}  \sum_{\bnu \in \cF_1} |g_{\bnu}| \omega_{\bnu},
\]
see \cite[proof of Theorem~6.1]{rawa13}. 
Altogether, with probability at least 
$1- \exp(-3m/(2s)) \geq 1- N^{-\log^3(s)}$ (by \eqref{m:samples})
\[
\left( \frac{1}{m} \sum_{\ell=1}^m |F_1(\bsy_\ell)|^2 \right)^{1/2} 
\leq 
\sqrt{\frac{5}{s}} \sum_{\bnu \in \cF_1} |g_{\bnu}| \omega_{\bnu} 
\leq \sqrt{\frac{5}{s}} \sigma_{s/2}(\mathbf{g})_{\omega,1}
\]
since the support of the weighted best $(s/2)$-sparse approximation to $\mathbf{g}$ 
is contained in $\Indx_0^s$ because no index $\bnu$ with $\omega_\bnu^2 \leq s/2$ is contained outside $\Indx_0^s$.
Furthermore, the weighted Stechkin estimate \eqref{weighted:Stechkin} 
together with $\mathbf{g} \in \ell_{\omega,p}$ implies
\[
\sigma_{s/2}(\mathbf{g})_{\omega,1} 
\leq 
2^{1/p-1} (s/2)^{1-1/p} \triple \mathbf{g} \triple_{\omega,p} 
= 
4^{1/p-1} \triple \mathbf{g} \triple_{\omega,p}
\]
so that
\[
\left( \frac{1}{m} \sum_{\ell=1}^m |F_1(\bsy_\ell)|^2 \right)^{1/2}  
\leq 
\sqrt{\frac{5}{s}} \sigma_{s}(\mathbf{g})_{\omega,1} 
\leq 
\sqrt{5}\cdot 4^{1/p-1} s^{1/2-1/p} \triple \mathbf{g} \triple_{\omega,p} 
\leq \varepsilon
\]
with probability at least $1-N^{-\log^3(s)}$ by Assumption~\eqref{eq:Condseps}. 
This implies that on this event
\[
\left( \frac{1}{m} \sum_{\ell=1}^m |b_\ell - F_0(\bsy_\ell)|^2 \right)^{1/2} \leq 2 \eps.
\]
It follows from Theorem~\ref{thm:weightRIP} that the weighted restricted isometry property of order $s$ of the sampling matrix $\bsA \in \R^{m \times N}$ in \eqref{sampling:matrix:Cheb} holds with probability 
at least $1-N^{-\log^3(s)}$ (precisely, $\delta_{\omega,2s} < 1/3$, say) since $m \geq C s \log^3(s) \log(N)$ by assumption. Then Theorem~\ref{thm:CS} implies that the reconstructed vector $\mathbf{g}^\sharp$ supported on $\Indx_0^s$
satisfies
\begin{align*}
\left(\sum_{\bnu \in \Indx_0^s} |g_\bnu - g^\sharp_{\bnu}|^p\right)^{1/2} & \leq C \frac{\sigma_s(\mathbf{g})_{\omega,1}}{\sqrt{s}} + C' \eps, \notag\\
\sum_{\bnu \in \Indx_0^s} |g_\bnu - g^\sharp_{\bnu}| \omega_{\bnu} & \leq C \sigma_s(\mathbf{g})_{\omega,1} + C' \sqrt{s}\, \eps. \notag
\end{align*}
Since the $T_{\bnu}$ form an orthonormal system we obtain for the $L^2$-error 
\begin{align*}
\| F - \widehat{F}_\eps \|_{L^2(U,\eta)} & \leq \| F_0 - \widehat{F}_\eps\|_{L^2(U,\eta)} + \|F_1\|_{L^2(U,\eta)} 
= \left(\sum_{\bnu \in \Indx_0^s} |g_\bnu - g^\sharp_{\bnu}|^2\right)^{1/2} +  \left(\sum_{\bnu \in \cF_1} |g_{\bnu}|^2 \right)^{1/2} \notag\\
&  \leq C \frac{\sigma_s(\mathbf{g})_{\omega,1}}{\sqrt{s}} + C' \eps + \sqrt{\frac{5}{s}} \sigma_{s/2}(\mathbf{g})_{\omega,1} \leq 
(C + \sqrt{5}) 2^{1/p-1} s^{-1/2} s^{1-1/p} \triple \mathbf{\mathbf{g}} \triple_{\omega,p} + C' \eps\\
&\leq C'' s^{1/2-1/p}  \triple \mathbf{g} \triple_{\omega,p}.
\end{align*}
Here, we have applied the weighted Stechkin estimate \eqref{weighted:Stechkin} 
as well as the upper bound of \eqref{eq:Condseps}.
Finally, since $\|T_{\bnu}\|_\infty \leq 2^{\|\bnu\|_0/2} \leq \omega_j$, the $L^\infty$-error 
can be bounded as follows:
\begin{align*}
\| F - \widehat{F}_\eps \|_{\infty} & \leq \| F_0 - \widehat{F}_\eps\|_{\infty} + \|F_1\|_{\infty} 
= \| \sum_{\bnu \in  \Indx_0^s} (g_{\bnu} - g_{\bnu}^\sharp) T_{\bnu} \|_{\infty} + \|\sum_{\bnu \in \cF_1} g_{\bnu} T_{\bnu} \|_{\infty}\\
&\leq \sum_{\bnu \in \Indx_0^s} |g_{\bnu} - g_{\bnu}^\sharp| \|T_{\bnu}\|_\infty + \sum_{\bnu \in \cF_1} |g_{\bnu}| \|T_{\bnu}\|_\infty
\leq \sum_{\bnu \in \Indx_0^s} |g_{\bnu} - g_{\bnu}^\sharp| \omega_{\bnu} + \sum_{\bnu \in \cF_1} |g_{\bnu}| \omega_{\bnu}\\
&\leq C \sigma_s(\mathbf{g})_{\omega,1} + C' \sqrt{s}\, \eps + \sigma_{s/2}(\mathbf{g})_{\omega,1}
\leq C'' s^{1-1/p} \triple g \triple_{\omega,p}.
\end{align*}
For the bound  $\sum_{\bnu \in \cF_1} |g_{\bnu}| \omega_{\bnu} \leq \sigma_{s/2}(g)_{\omega,1}$ we have used again that by the definition of $\Indx_0^s$ the support of the weighted
best $s/2$-term approximation is contained in $\Indx_0^s$. 
The second inequalities in \eqref{err:bound:L2} and \eqref{err:bound:Linf} follow from $m \asymp s \log^3(s) \log(N)$.
This completes the proof.
\end{proof}

In order make the estimate $m \geq C s \log^3(s) \log(N)$ on the number of sample evaluations 
precise, we also need a bound on  $N = \#\Indx_0^s$ (depending on $s$).
The size of $\Indx_0^s$ will also determine the complexity of the weighted $\ell_1$-minimization step.
If the weight sequence $v_j$ grows to infinity as $j \to \infty$, 
then the number of active indices $j$ which are relevant for $\Indx_0^s$, i.e., 
the number 
$d := \max\{j: \exists \bnu \in \Indx_0^s \mbox{ with } \nu_j \neq 0 \}$ is finite. 
In fact, if the sequence $v_j$ is monotonically increasing, then
$d = \max\{j : v_j \leq \sqrt{s/4}\}$. 
Therefore, setting $t = s/2$, it suffices to bound the size of
\[
\Indx_d(t,\mathbf{v}) := \{ \bnu \in \N_0^d : 2^{\|\bnu\|_0} \prod_{j \in \supp \bnu} v_j^{2\nu_j} \leq t \}
\;.
\]
\begin{thm}\label{thm:indx:size} 
For a finite weight sequence $\mathbf{v} = (v_1,\hdots,v_d)$ with $v_j > 1$, 
$a_j := 2 \log_2(v_j)$ and $A := \log_2(t)$, we have
\begin{equation}\label{Indx:Size:est}
\#\Indx_d(t,\mathbf{v}) \leq 1 + \sum_{k=1}^{\min\{d, \lfloor A \rfloor\}} \frac{(A-k)^k}{k!} \sum_{\substack{S \subset [d]\\ \#S=k\\ \sum_{j \in S} a_j \leq A-k}}  \prod_{\ell \in S} a_\ell^{-1}. 
\end{equation}
\end{thm}
\begin{proof} Taking the base-$2$ logarithm, the defining condition of the set $\Indx_d(t,\mathbf{v})$ becomes
\[
\|\bnu\|_0 + \sum_{j \in \supp \bnu} 2 \log_2(v_j) \nu_j \leq \log_2(t),
\]
which by the definition of $a_j$ and $A$ is equivalent to
\begin{equation}\label{cond:bnu}
\sum_{j \in \supp \bnu} (1 + a_j \nu_j) \leq A.
\end{equation}
In order to estimate the number of points $\bnu$ with nonnegative integer coordinates satisfying this inequality,
we introduce, for $L >0$ and $\mathbf{b} = (b_1,\hdots,b_k)$ with $b_\ell > 0$, the cardinality 
of a set of multi-indices
for which all components
of its members $\bnu$ are at least $1$: 
$$\Gamma_k(L,\mathbf{b}) := \#\{\bnu \in \N^{k}: \sum_{j=1}^{k} b_j \nu_j \leq L \}.
$$
Observe that $\Gamma_k(L,\mathbf{b}) = 0$ if $\sum_{j=1}^{k} b_j > L$. Moreover, for $\bnu$ with $\|\bnu\|_0 = k$, condition
\eqref{cond:bnu} becomes $\sum_{j \in \supp \bnu} a_j \bnu_j \leq A - k$. Considering all possible support sets $S = \supp \bnu \subset [d]$ of cardinality $k$ and summing
over all possible $k$ we obtain
\begin{equation}\label{Indx:size}
\# \Indx_d(t,\mathbf{v}) = 1 + \sum_{k=1}^{\min\{d,\lfloor A \rfloor\}} \sum_{\substack{S \subset [d]\\ \#S=k
}} \Gamma_{k}(A-k,(a_j)_{j \in S}).
\end{equation}
If $k > \lfloor A \rfloor$ then $\Gamma_{k}(A-k,(a_j)_{j \in S}) = 0$, which is why the first sum 
above runs only up to $\lfloor A \rfloor$ if $\lfloor A \rfloor < d$. We claim that
\begin{equation}\label{Gamma:size}
\Gamma_k(L, \mathbf{b}) \leq \frac{L^k}{k! \prod_{j =1}^k b_j}.
\end{equation}
We show this by induction on $k$. For $k=1$, we have 
\[
\Gamma_1(L,b_1) = \#\{ \nu_1 \in \N : b_1 \nu_1 \leq L\} = \lfloor L/b_1 \rfloor \leq L/b_1.
\]
Next, assume the claim holds for $k \in \N$. 
Then
\begin{align*}
\Gamma_{k+1}(L,b_1,\hdots,b_{k},b_{k+1}) & = 
\sum_{\nu_{k+1}=1}^{\lfloor L/b_{k+1} \rfloor} \Gamma_k(L- b_{k+1} \nu_{k+1},b_1,\hdots,b_k) 
\leq \sum_{\nu_{k+1}=1}^{\lfloor L/b_{k+1} \rfloor} \frac{(L-b_{k+1} \nu_{k+1})^k}{k! \prod_{j=1}^k b_j} 
\\
& \leq  \frac{1}{k! \prod_{j=1}^k b_j}  \int_0^{L/b_{k+1}} (L-b_{k+1}x)^k dx = \frac{1}{k! \prod_{j=1}^k b_j} \frac{L^{k+1}}{(k+1)b_{k+1}} 
\\
& = \frac{L^{k+1}}{(k+1)! \prod_{j =1}^{k+1} b_j}.
\end{align*}
Here, we have applied a simple comparison of a sum and an integral. 
This shows the claimed inequality \eqref{Gamma:size}.
Recalling that $\Gamma_k(A-k,(a_j)_{j \in S}) = 0$ for $\#S = k$ if $\sum_{j \in S} a_j > A-k$, 
Equation~\eqref{Indx:size} yields the claimed estimate. 
\end{proof}
With the same technique a lower bound can be shown as well. 
We illustrate the estimate of the theorem with some important examples of weights. The proofs of these estimates will be provided in the appendix.
\begin{cor}\label{cor:size:Indx} 
Assume $s \geq 1$.
\begin{itemize}
\item[(a)] Let $v_j = \beta$, $j=1,\hdots,d$, for some $\beta > 1$ and $v_j = \infty$ for $j \geq d$ 
(so that the sum in \eqref{eq:Baffine} representing $A(\bsy)$ involves only $d$
nonzero terms). Then
\begin{equation*}
\# \Indx_0^s \leq \left\{\begin{array}{ll} \left(\log_{\beta^2}(\beta^2 s/2)\right)^d  & \mbox{ if } d \leq \log_{2\beta^2}(s/2), \\
\left(\left(1+\frac{1}{\log_2(\beta^2)}\right)ed\right)^{\log_{2\beta^2}(s/2)}
& \mbox{ otherwise}.
\end{array} \right.
\end{equation*}
\item[(b)] Let $v_j = c j^\alpha$, $j = 1,2,\hdots$, grow polynomially 
for some $c > 1$ and $\alpha > 0$. Then there are constants $C_{\alpha,c}> 0$ and $\gamma_{\alpha,c} > 0$ such that
\[
\#\Indx_0^s \leq C_{\alpha,c} s^{\gamma_{\alpha,c} \log(s)}.
\]
\item[(c)] Let $v_j = \beta^j$ for some $\beta > 1$. Then
\begin{equation}\label{size:exponential:weights}
\Indx_0^s \leq 1 +\frac{1}{2\pi\sqrt{\log_\beta(s/2)}} \left(e^3 \sqrt{\log_\beta(s/2)}\right)^{\sqrt{\log_\beta(s/2)}}.
\end{equation}
\end{itemize}
\end{cor}

\begin{remark}\label{rem:Indx:size} We briefly illustrate consequences for the number $m \geq C s \log^3(s) \log(N)$ of samples
required in Theorem~\ref{thm:main:func}.
\begin{itemize}
\item[(a)] For constant weights $v_j = \beta > 1$, $j=1,\hdots,d$, the resulting bound becomes
\[
m \geq \left\{ \begin{array}{ll} C_\beta \log(d) s \log^4(s) & \mbox{ if } s \leq (2 \beta)^{2d}\\
C_\beta d s \log^3(s) \log(\log(s))  & \mbox{ if } s > (2 \beta)^{2d}.
\end{array} \right.
\]
\item[(b)] For polynomially growing weights $v_j = c j^\alpha$ with $c > 1$ and $\alpha > 0$ we need
\[
m \geq C_{\alpha,c} s \log^5(s).
\]
Moreover, the resulting error bound \eqref{err:bound:Linf} becomes $\|F-\widehat{F}_\eps\|_\infty \leq C_{\alpha,c} \|\mathbf{g}\|_{\omega,p} (\log^5(m)/m)^{1/p-1}$, and likewise for the
error bound in $L^2$.
\item[(c)] 
For exponentially growing weights $v_j = \beta^j$ with $\beta > 1$ 
the upper bound \eqref{size:exponential:weights} 
on $\#\Indx_0^s$ grows slowlier than $s$.
This means that $m$ may actually be chosen smaller than 
$N = \#\Indx_0^s$ which leads to an overdetermined linear system. 
In this scenario,
the application of compressive sensing may not be required and it is likely good enough to 
use least-squares methods \cite{CDL2013,MiNoSchT2014} to compute coefficients $\mathbf{b}$ from sample
evaluations. We postpone a detailed discussion to a later contribution.
\end{itemize}
\end{remark}
\section{Conclusions}
\label{sec:Concl}
In the present paper, we developed a convergence theory of 
compressed sensing based approximations of solution
functionals for
high-dimensional, parametric operator equations.
Such problems arise generically in numerical uncertainty 
quantification, when laws of random field inputs are described by 
countably-parametric, deterministic functions, and the response,
i.e., the parametric solution of the 
affine-parametric operator equation \eqref{eq:Baffine}, 
is to be approximated in terms of the countably many parameters.  
As main contribution we showed that given $m$ (approximate) sample evaluations
at randomly chosen parameter points, the computed parametric solution $\widehat{F}$ 
approximates the true function $F(\bsy) = G(u(\bsy))$ at rates
\[
\|F - \hat{F}\|_{L_2} \leq C \left( \frac{\log(m)^5}{m} \right)^{1/p-1/2}, \qquad 
\|F - \hat{F}\|_{L_\infty} \leq C \left( \frac{\log(m)^5}{m} \right)^{1/p-1},
\]
under a weighted $\ell_p$-summability assumption with $0 < p < 1$ 
on the input parameters (for polynomially growing weights $v_j$, 
see Section~\ref{sec:CS} for details).
Up to the logarithmic terms, this matches the rates of best 
$s$-term approximation and for $p<1/2$ outperforms the 
(best possible) convergence rate for Monte-Carlo methods. 

As an important ingredient, potentially of independent interest, 
we generalized available $\ell_p$-estimates for polynomial
chaos expansions \cite{chalch14,codesc10-1,codesc11,Hanse486,Kuo475,ScMCQMC12} 
to the weighted case, 
which allowed to apply recent results on weighted sparsity and recovery via weighted
$\ell_1$-minimization \cite{rawa13}. 
These weighted estimates also allow to 
determine good a-priori choices for a finite subset of $\cF$ containing 
the relevant \Tscheb coefficients -- the set $\Indx_0^s$ defined in \eqref{def:Indxs}. 
The actual support set of the best 
$s$-term approximation may be much smaller than $\Indx_0^s$,
but such an initial choice is required in order to run 
weighted $\ell_1$-minimization for the reconstruction.
We note that $\Indx_0^s$ may be a good choice also for other 
approaches such as least squares methods \cite{CDL2013,MiNoSchT2014}, 
see Remark~\ref{rem:Indx:size}(c).
We developed our theory for the affine-parametric operator equations
\eqref{eq:main2}; we point out, however, that the compressive sensing
part of our analysis is independent of the particular 
parameter dependence: the proof of
$p$-summability of the \Tscheb coefficients of the parametric 
solution given in Sections \ref{sec:AnCont},
\ref{sec:BndTscheb} only required the existence of holomorphic 
extension of parametric solutions into polydiscs. Such extensions
are available for several more general classes of parametric problems:
for affine-parametric, nonlinear initial value ODEs \cite{Hanse1085},
for more general, nonlinear countably-parametric operator equations
\cite{chalch14} (including in particular also uncertain domain
parametrizations). 
We expect the presently developed methods
to apply also for these more general classes of parametric
operator equations.

Let us place the present results into perspective with 
other approximation methods for high-dimensional problems: 
adaptive stochastic Galerkin methods \cite{EGSZ14_1230,EGSZ14_1045,G13_531},
reduced basis approaches \cite{BCDDPW11,BufMadPat12} 
adaptive Smolyak discretizations \cite{SS14_1132,SS14_1117}
and adaptive interpolation methods \cite{CCS2014} 
are all \emph{sequential} in the sense that they rely on \emph{successive} 
numerical solution of the operator equation on parameter instances 
(adaptive Galerkin discretizations being intrusive on top).
This is in contrast to, say, Monte-Carlo or Quasi-Monte Carlo approaches
which likewise offer dimension-independent convergence rates for statistical
moments of the solution, and which allow to access the parametric solution
\emph{simultaneously} at a set of samples. These methods do not, however,
allow \emph{recovery} of the parametric solution 
as does the presently proposed approach. 
Moreover, the convergence rates of MC are well-known to be limited to $1/2$.

We have discussed the approximation of functionals of parametric solutions. MC methods, for instance, are
also able to compute expectations of the full solution $\E_{\bsy}[u(\bsy)]$, not only $\E_{\bsy}[G(u(\bsy))]$, i.e., 
expectations of a functional of the solution. 
Extensions of our method to the computation of parametric solutions $u(\bsy)$ that are fully resolved in the physical domain
will be the subject of future investigations. We strongly expect that we will again obtain near-optimal convergence rates 
so that the resulting method shows the same advantages with respect to MC and other methods outlined above for the
case when functionals of the solutions are computed.

In contrast to the above mentioned sequential and deterministic methods, 
our CS-based methods are similar in nature to
\emph{so-called least-squares projection methods}, recently proposed 
in \cite{CDL2013,MiNoSchT2014}, which also rely on sample evaluations at randomly chosen
parameter locations. These least squares approaches require
\emph{a-priori}  knowledge of a near-optimal set of active indices
in the gpc expansions of the parametric solution and have complexity exceeding 
the cardinality of that set. In contrast, the presently proposed CS approach requires only 
specification of a (conservative and possibly large) superset $\Indx_0^s$ 
containing the optimal $s$-term approximation set. The complexity of the CS approach is, as we showed
in the present article, \emph{sublinear} in the cardinality of this
candidate set and will pinpoint the `essential'
gpc coefficients within this set, even if the set of these coefficients
is ``gappy'' or lacunary.  Expressed differently, given a budget $m$ of samples, least squares methods
need a good guess for an index set of cardinality somewhat smaller than $m$ of relevant gpc coefficients, while
compressive sensing methods need only a very rough estimate for this index set which is allowed to be significantly
larger than $m$.

In the present paper, we considered only 
affine parametric operator equations. 
We expect that extensions to nonaffine, but still holomorphic, 
parameter dependence is possible analogously to \cite{chalch14},
which will allow the application to \emph{Bayesian inverse problems} for parametric operator equations \cite{SS13_813,SS14_1117}.
These topics will be developed in detail elsewhere.

\section{Appendix: Proof of Corollary~\ref{cor:size:Indx}}

(a) Set $a_j = 2 \log_2(v_j) = \log_2(\beta^2)$ and $A = \log_2(s/2)$. Observe that condition $\sum_{j \in S} a_j \leq A -k$ in 
the second sum in \eqref{Indx:Size:est} can be met for some $S$ with $\#S = k$ if and only if $k \leq A/(1+\log_2(\beta^2))$.
Let us first consider the case that $d \leq A/(1+\log_2(\beta^2))= A/\log_2(2\beta^2)= \log_{2\beta^2}(s/2)$. 
Then clearly $d \leq A$ and \eqref{Indx:Size:est} yields
\begin{align*}
\#\Indx_d(t,\mathbf{v}) & \leq 1 + \sum_{k=1}^{d} \frac{(A-k)^k}{k!} \sum_{\substack{S \subset [d]\\ \#S=k\\ \sum_{j \in S} a_j \leq A-k}}  \prod_{j \in S} a_j^{-1}
= 1 + \sum_{k=1}^d  \frac{(A-k)^k}{k!} \binom{d}{k} (\log_2(\beta^2))^{-k} \\
& \leq \sum_{k=0}^d  \binom{d}{k} (A/\log_2(\beta^2))^k = 
 (A/\log_2(\beta^2)+1)^d = (\log_{\beta^2}(\beta^2 s/2))^d.  
 \end{align*}
On the other hand, if $d > A/\log_2(2\beta^2)$ then we have
\begin{align*}
\#\Indx_d(t,\mathbf{v}) & \leq \sum_{k=0}^{\lfloor A/\log_2(2\beta^2) \rfloor} \frac{(A-k)^k}{k!} \binom{d}{k} (\log_2(\beta^2))^{-k}  
 \leq  \sum_{k=0}^{\lfloor A/\log_2(2\beta^2) \rfloor} \binom{d}{k} (A/\log_2(\beta^2))^{k} \\
 &\leq (A/\log_2(\beta^2))^{\lfloor A/\log_2(2\beta^2) \rfloor} \sum_{k=0}^{\lfloor A/\log_2(2\beta^2) \rfloor} \binom{d}{k}
 \leq (A/\log_2(\beta^2))^{\lfloor A/\log_2(2\beta^2) \rfloor} \left(\frac{ed}{\lfloor A/\log_2(2\beta^2) \rfloor}\right)^{\lfloor A/\log_2(2\beta^2) \rfloor}\\
 & \leq (\log_2(2\beta^2)ed/\log_2(\beta))^{A/\log_2(2\beta^2)} = \left((1+1/\log_2(\beta^2))ed\right)^{\log_2(s/2)/\log_2(2\beta^2)} \\
 &= ((1+1/\log_2(\beta^2))ed)^{\log_{2\beta^2}(s/2)}.
\end{align*}
Here, we have applied the inequality $\sum_{k=0}^n \binom{d}{k} \leq (ed/n)^n$, see e.g.\ \cite[Theorem 3.7]{anba99}, 
and used the fact that $x \mapsto (K/x)^{x}$ is monotonically
increasing for $0<x \leq K/e^{1/K}$.

(b) We first note that the largest effective component of the indices in $\Indx_0^s$ is given by
\[
d = \max\{j: \exists \bnu \in \Indx_0^s \mbox{ with } \nu_j \neq 0 \} = \max\{j : v_j \leq \sqrt{s/4}\} = \lfloor \left(s/(4c) \right)^{1/(2\alpha)}\rfloor.
\]
We set $a_j = 2\log_2(v_j) = 2 \alpha \log_2(j) + 2 \log_2(c)$, $A = \log_2(s/2)$. 
The estimate \eqref{Indx:Size:est} gives then
\begin{align*}
\#\Indx_0^s & \leq 1 
+ \sum_{k=1}^{\lfloor A \rfloor} \frac{(A-k)^k}{k!} \sum_{\substack{S \subset [d]\\ \#S=k\\ \sum_{j \in S} a_j \leq A-k}}  \prod_{j \in S} a_j^{-1} \\
& =  1 
+ \sum_{k=1}^{\lfloor A \rfloor} \frac{(A-k)^k}{k!} \sum_{\substack{S \subset [d]\\ \#S=k\\ 2 \alpha \log_2(k!)  \leq A-k(1+2 \log_2(c))}}  \prod_{j \in S} \frac{1}{2\alpha \log_2(j) + 2 \log_2(c)} \\
& \leq 1 + \sum_{\substack{k \in [\lfloor A \rfloor] \\ 2 \alpha \log_2(k!)  \leq A-k(1+2 \log_2(c))}} \frac{(A-k)^k}{k!} 2^{-k} \sum_{\substack{S \subset [d]\\ \#S=k}} (\alpha + \log_2(c))^{-(k-1)} (\log_2(c))^{-1} \\
& \leq 1 + (1+ \alpha/\log_2(c)) \sum_{\substack{k \in [\lfloor A \rfloor] \\ 2 \alpha \log_2(k!)  \leq A-k(1+2 \log_2(c))}}
\left(\frac{A}{2\alpha + 2\log_2(c)}\right)^{k} \frac{1}{k!} \binom{d}{k}.
\end{align*}
Let $K$ be the maximal number $k$ such that $2 \alpha \log_2(k!)  \leq A-k(1+2 \log_2(c))$. For $k\geq 4$ 
we have $\log_2(k!) \geq k$. (The case $K \leq 3$ only occurs for $s < D_c$, which can be handled by potentially adjusting constants.)
Therefore, $K \leq \frac{A}{1+2\alpha + 2 \log_2(c)} $. 
Moreover, Stirling's inequality gives $k! \leq \sqrt{2\pi k} e^{1/(12k)} (k/e)^k$ 
for $k\geq 4$ so that for such $k$ we have
$\log_2(k!) \leq \log_2(2 \pi e^{1/24}k)/2 + k \log_2(k/e) 
  \leq \left(\frac{\log_2(2 \pi e^{1/24}\cdot 4)}{8} + \log_2(k/e)\right) k
= \log_2(D_0 k) k$ with $D_0 = (8\pi e^{1/24})^{1/8}/e \approx 0.55$.
Therefore, $K$ also satisfies the lower bound
$K \geq \frac{A}{1+ 2\alpha+ \alpha \log_2(D_0 K)} \geq  \frac{A}{1+\alpha \log_2(D_1 A)}$ with $D_1 = 2 D_0 = 1.1067$.
Furthermore, observe that, for $k = 1,\hdots,K$, 
the sequence $b_k = \left(\frac{A}{2\alpha + 2\log_2(c)}\right)^{k} \frac{1}{k!}$ 
takes the maximum for $k = K$ so that, exploiting the Stirling inequality 
$K! \geq \sqrt{2\pi K} (K/e)^K$, 
we obtain
\begin{align*}
\#\Indx_0^s &\leq \left(1+ \frac{\alpha}{\log_2(c)}\right) \left(\frac{A}{2\alpha + 2\log_2(c)}\right)^{K} \frac{1}{K!} \sum_{k=0}^K \binom{d}{k} 
\leq  \left(1+ \frac{\alpha}{\log_2(c)}\right)  \left(\frac{A}{2\alpha + 2\log_2(c)}\right)^{K} \frac{1}{\sqrt{2\pi K} (K/e)^K} \left(ed/K\right)^K\\
& \leq  \left(1+ \frac{\alpha}{\log_2(c)}\right)  \frac{1}{\sqrt{2\pi K}}   \left(\frac{e^2 d A}{K^2(2\alpha + 2\log_2(c))}\right)^{K}\\
& \leq  \left(1+ \frac{\alpha}{\log_2(c)}\right)  \sqrt{\frac{1+\alpha \log_2(D_1 A)}{2 \pi A}}\left( \frac{e^2 d}{A} \cdot \frac{1+\alpha\log_2(D_1 A)}{2\log_2(2^\alpha c)}\right)^{A/(1+\log_2(2^{\alpha} c))}.
\end{align*}
Inserting here $A= \log_2(s/2)$ 
yields
\[
\#\Indx_0^s \leq  \left(1+ \frac{\alpha}{\log_2(c)}\right)  \sqrt{\frac{1+\alpha \log_2(D_1 \log_2(s/2))}{2 \pi \log_2(s/2)}} 
\left( \frac{e^2\alpha d}{2 \log_2(s/2)} \cdot \frac{\log_2(2^{1/\alpha}D_1 A)}{\log_2(2^\alpha c)}\right)^{\log_2(s/2)/\log_2(2^{\alpha+1} c)}.
\]
Further using that $d \leq (s/4c)^{1/\alpha}$ yields
\begin{align*}
\#\Indx_0^s & \leq 
\left(1+ \frac{\alpha}{\log_2(c)}\right)  \sqrt{\frac{1+\alpha \log_2(D_1 \log_2(s/2))}{2 \pi \log_2(s/2)}} 
\left( \frac{e^2\alpha (s/4c)^{1/(2\alpha)}}{2 \log_2(s/2)} \cdot  \log_{2^{\alpha}c} (2^{1/\alpha}D_1 \log_2(s/2))\right)^{\log_{2^{\alpha+1}c}(s/2)}\\
& \leq C_{\alpha,c} s^{\gamma_{\alpha,c} \log(s)}
\end{align*}
for suitable constants $C_{\alpha,c}, \gamma_{\alpha,c}>0$ depending only on $\alpha$ and $c$.

(c) For exponentially growing weights $v_j = \beta^j$ we have  
\[
d = \max\{j: \exists \bnu \in \Indx_0^s \mbox{ with } \nu_j \neq 0 \} = \max\{j : v_j \leq s/4\} = \lfloor \frac{\log_2(s/4)}{\log_2(\beta)} \rfloor.
\]
We set $a_j = 2 \log_2(v_j) =  2j \log_2(\beta)$, $A = \log_2(s/2)$. 
The estimate \eqref{Indx:Size:est} now reads
\begin{align*}
\#\Indx_0^s & \leq 1 
+ \sum_{k=1}^{\min\{d,\lfloor A \rfloor\}} \frac{(A-k)^k}{k!} \sum_{\substack{S \subset [d]\\ \#S=k\\ \sum_{j \in S} a_j \leq A-k}}  \prod_{j \in S} a_j^{-1} \\
& \leq 1 + \sum_{\substack{k=1\\ 2\log_2(\beta) \sum_{j=1}^k j \leq A-k}}^{\min\{d,\lfloor A \rfloor\}} \frac{(A-k)^k}{k!} 
\sum_{\substack{S \subset [d]\\ \#S=k}}  (2 \log_2(\beta))^{-k} \prod_{j \in S} j^{-1}. \\
\end{align*}
Note that since $\sum_{j=1}^k j = k(k+1)/2 \geq k^2/2$, 
the condition $2\log_2(\beta) \sum_{j=1}^k j \leq A-k$ 
is implied by $k \leq \sqrt{\frac{A}{\log_2(\beta)}} < d$. 
Therefore,
\begin{align*}
\#\Indx_0^s  & \leq 1 + \sum_{k=1}^{\lfloor \sqrt{A/\log_2(\beta)} \rfloor} \left(\frac{A}{2 \log_2(\beta)}\right)^k \binom{d}{k} \frac{1}{(k!)^2} 
\\
& \leq 1 + 
\left(\frac{A}{2\log_2(\beta)}\right)^{\lfloor \sqrt{A/\log_2(\beta)}\rfloor} \frac{1}{(\lfloor \sqrt{A/\log_2(\beta)}\rfloor!)^2} 
\sum_{k=1}^{\lfloor \sqrt{A/\log_2(\beta)} \rfloor} \binom{d}{k} 
\\
& \leq 1 +  \left(\frac{A}{2\log_2(\beta)}\right)^{\sqrt{A/\log_2(\beta)}} \frac{1}{2 \pi \sqrt{A/\log_2(\beta)} (A/\log_2(\beta)/e^2)^{\sqrt{A/\log_2(\beta)}}} 
 \left(\frac{ed}{\sqrt{A/\log_2(\beta)}}\right)^{\sqrt{A/\log_2(\beta)}} 
\\
& = 1 + \frac{1}{2\pi \sqrt{A/\log_2(\beta)}} \left(\frac{e^3d}{\sqrt{A/\log_2(\beta)}}\right)^{\sqrt{A/\log_2(\beta)}} \\
& \leq 1 + \frac{1}{2\pi \sqrt{\log_2(s/2)/\log_2(\beta)}} \left(\frac{e^3 \log_2(s/4)/\log_2(\beta)}{\sqrt{\log_2(s/2)/\log_2(\beta)}}\right)^{\sqrt{\log_2(s/2)/\log_2(\beta)}} 
\\
 & \leq 1 +\frac{1}{2\pi\sqrt{\log_\beta(s/2)}} \left(e^3 \sqrt{\log_\beta(s/2)}\right)^{\sqrt{\log_\beta(s/2)}}.
 \end{align*}
This concludes the proof.

\section*{Acknowledgements}
HR~would like to thank Albert Cohen for enlightening discussions on solving parametric 
PDEs and for his warm hospitality during a stay at
Universit{\'e} Pierre et Marie Curie in Paris in 2010. HR is also grateful for the hospitality 
of the Seminar of Applied Mathematics at ETH Zurich during a stay in spring 2011. 
HR acknowledges funding from the Hausdorff Center of Mathematics at the University of Bonn 
and from the European research council (ERC) 
via the Starting Grant StG 258926. CS is supported by the ERC via the Advanced Grant AdG 247277. 
HR and CS acknowledge the support of ICERM at Brown University during the final stages 
of preparation of this manuscript, 
during the special semester on High-dimensional approximation in September 2014.

\bibliographystyle{amsplain}
\bibliography{CSPDE}
\end{document}